\documentclass{article}

\usepackage{graphicx}
\usepackage{multicol}
\usepackage{amsmath}
\usepackage{amssymb}
\usepackage{amsthm}
\usepackage{mathabx}
\usepackage{stmaryrd}
\usepackage{txfonts}
\usepackage{pstricks,pst-node}
\usepackage{enumerate}
\usepackage{centernot}
\usepackage{changepage}

\newtheorem{lemma}{Lemma}[section]
\newtheorem{theorem}[lemma]{Theorem}
\newtheorem{proposition}[lemma]{Proposition}
\newtheorem{corollary}[lemma]{Corollary}
\newtheorem{definition}[lemma]{Definition}

\newtheorem*{theorem*}{Theorem}

\theoremstyle{definition}
\newtheorem{examples}[lemma]{Examples}
\newtheorem{example}[lemma]{Example}

\theoremstyle{remark}
\newtheorem{remark}[lemma]{Remark}

\newcommand{\down}{\mathop{\downarrow}}

\newcommand{\ascategory}[1]{}

\newcommand{\DeclareCategory}[2]{\expandafter\newcommand{#1}{{ifmmode{}}}}
\DeclareMathOperator{\ob}{ob}

\def\Rel{\ifmmode{\mathcal R}el\else ${\mathcal R}el$\fi}
\def\SRel{\ifmmode{\mathcal SR}el\else ${\mathcal SR}el$\fi}
\def\FinSet{\ifmmode{\mathcal F}inSet\else ${\mathcal F}inSet$\fi}
\def\BetRel{\ifmmode{\mathcal B}etRel\else ${\mathcal B}etRel$\fi}
\def\CCDLat{\ifmmode{\mathcal C}CD\textrm{-}Lat\else ${\mathcal C}CD\textrm{-}Lat$\fi}
\def\Convex{\ifmmode{\mathcal C}onvex\else ${\mathcal C}onvex$\fi}
\def\Preconvex{\ifmmode{\mathcal P}\!reconvex\else ${\mathcal P}\!reconvex$\fi}
\def\ConvexTop{\ifmmode{\mathcal C}onvex{\mathcal T}\!\!op\else ${\mathcal C}onvex{\mathcal T}\!\!op$\fi}
\def\Inf{\ifmmode{\mathcal I}\!n\!f\else ${\mathcal I}\!n\!f$\fi}
\def\Ord{\ifmmode{\mathcal O}rd\else ${\mathcal O}rd$\fi}
\def\Sup{\ifmmode{\mathcal S}up\else ${\mathcal S}up$\fi}
\def\PartialSup{\ifmmode{\mathcal P}artial{\mathcal S}up\else
  ${\mathcal P}artial{\mathcal S}up$\fi}
\def\DistPartSup{\ifmmode{\mathcal D}ist{\mathcal P}art{\mathcal
    S}up\else ${\mathcal D}ist{\mathcal P}art{\mathcal S}up$\fi}

\newcommand{\Topol}{\ensuremath{{\mathcal T}\!op}}

\newcommand{\Frame}{\ensuremath{\mathcal F}\!\!rame}
\newcommand{\Coframe}{\ensuremath{\mathcal C}o\!f\!rame}
\newcommand{\SpatialCoframe}{\ensuremath{\mathcal S}\mspace{-1mu}patial{\mathcal C}o\!f\!rame}
\newcommand{\TC}{\ConvexTop}

\newcommand{\TCGPartialSup}{\ensuremath{\mathcal
    TCGP}\mathit{artial}{\mathcal S}\mathit{up}}

\def\Euclid{\ifmmode{\mathcal E}uclidean\else ${\mathcal E}uclidean$\fi}
\def\ConvexMetrisable{\ifmmode{\mathcal C}onvexMetrisable\else ${\mathcal C}onvexMetrisable$\fi}

\def\CTop{\ifmmode{\mathcal CT}op\else ${\mathcal CT}op$\fi}
\def\Metric{\ifmmode{\mathcal M}etric\else ${\mathcal M}etric$\fi}
\def\Set{\ensuremath{\rm\bf Set}}
\def\op{^{\rm op}}
\newcommand{\co}{^{\rm co}}
\def\C{\ensuremath{\mathcal C}}

\def\lmorphism #1{\xymatrix@1{\ar@{->}[rr]^{#1}&&}}

\def\FinPart{\ifmmode{\mathcal F}inPart\else ${\mathcal F}inPart$\fi}
\def\Linear{\ifmmode{\mathcal L}inear\else ${\mathcal L}inear$\fi}

\def\mmorphism #1{\xymatrix@1{\ar@{ >->}[r]^{#1}&}}
\def\emorphism #1{\xymatrix@1{\ar@{->>}[r]^{#1}&}}
\newcommand{\mono}{\mmorphism{}}

\def\kmorphism#1{\xymatrix@1{\ar[r]^{#1}|\vert&}}

\def\pmorphism#1{\xymatrix@1{\ar@_{->}[r]^{#1} &}}
\newcommand{\parmorphism}[1]{\xymatrix@1{\ar@{=>}[r]|\circ^{#1} &}}
\newcommand{\dmorphism}[1]{\xymatrix@1{\ar@{->}[r]|\circ^{#1} &}}

\input xy
\xyoption{all}
\SelectTips{cm}{}
\newdir{ >}{{}*!/-5pt/@{>}}
\newdir{|>}{!/2pt/@{|}*:(6,-.2)@^{>}*:(2,+.2)@_{>}}

\def\morphism#1{\xymatrix@1{\ar[r]^{#1}&}}
\def\cmorphism #1{\xymatrix@1{\ar@{-|>}[r]^{#1}&}}
\newcommand{\dcmorphism}[1]{\xymatrix@1{\ar@{-|>}[r]|\circ^{#1} &}}

\begin{document}


\title{Stone Duality for Topological Convexity Spaces} 
\author{Toby~Kenney \\ \\
Dalhousie University}
\maketitle

\begin{abstract}
  A convexity space is a set $X$ with a chosen family of subsets
  (called convex subsets) that is closed under arbitrary intersections
  and directed unions. There is a lot of interest in spaces that have
  both a convexity space and a topological space structure. In this
  paper, we study the category of topological convexity spaces and
  extend the Stone duality between coframes and topological spaces to
  an adjunction between topological convexity spaces and
  sup-lattices. We factor this adjunction through the category of
  preconvexity spaces (sometimes called closure spaces).   

\end{abstract}

\maketitle

\subsection*{Keywords}

Stone duality; Topological Convexity Spaces; Sup-lattices;
Preconvexity Spaces; Partial Sup-lattices

\section{Introduction}

Stone duality is a contravariant equivalence of categories between
categories of spaces and categories of lattices. The original Stone
duality was between Stone spaces and Boolean algebras
\cite{Stone}. One of the most widely used extensions of Stone duality
is between the categories of sober topological spaces and spatial
coframes (or frames --- since this is a 1-categorical duality, they
are the same thing). This duality extends to an idempotent adjunction between
topological spaces and coframes, given by the functors that send a
topological space to its coframe of closed sets, and the functor that
sends a coframe to its space of points.

In this paper, we develop an idempotent adjunction between topological
convexity spaces and sup-lattices (the category whose objects are
complete lattices, and morphisms are functions that preserve arbitrary
suprema). Topological convexity spaces are sets equipped with both a
chosen family of closed sets and a chosen family of convex sets. A
canonical example is a metric space $X$ with the usual metric
topology, and convex sets being sets closed under the
betweenness relation given by $y$ is between $x$ and $z$ if
$d(x,z)=d(x,y)+d(y,z)$. Many of the properties of metric spaces extend
to topological convexity spaces. Homomorphisms of topological
convexity spaces are continuous functions for which the inverse image
of a convex set is convex.

Our approach to showing this adjunction goes via two equivalent
intermediate categories. The first is the category of preconvexity
spaces. A preconvexity space is a pair $(X,{\mathcal P})$ where
${\mathcal P}$ is a collection of subsets of $X$ that is closed under
arbitrary intersections and empty unions. We will refer to sets
$P\in{\mathcal P}$ as {\em preconvex} subsets of $X$. A homomorphism
of preconvexity spaces $f:(X,{\mathcal P})\morphism{}(X',{\mathcal
  P}')$ is a function $f:X\morphism{}X'$ such that for any $P\in
{\mathcal P}'$, we have $f^{-1}(P)\in{\mathcal P}$. This category of
preconvexity spaces was also studied by~\cite{Dawson1987}, and shown
to be closed under arbitrary limits and colimits.

The second intermediate category that is equivalent to the category of
preconvexity spaces, is a full subcategory category of Distributive
Partial Sup lattices. This category was studied
in~\cite{PartialSup}. Objects of this category are complete lattices
with a chosen family of suprema which distribute over arbitrary
infima. Morphisms are functions that preserve all infima and the
chosen suprema. The motivation for partial sup lattices was an
adjunction between partial sup lattices and preconvexity spaces, which
is shown in~\cite{PartialSup}.

Before we begin presenting the extension of Stone duality to
topological convexity spaces, Section~\ref{Preliminaries} provides a
review of the main ingredients needed. While these reviews do not
contain substantial new results, they are presented in a with a
different focus from much of the literature, so we hope that the
reviews offer a new perspective on these well-studied subjects.  We
first recap the basics of topological convexity spaces. We then review
Stone duality for topological spaces. We then review the category of
distributive partial sup-lattices. This category was defined
in~\cite{PartialSup}, with the motivation of modelling various types of
preconvexity spaces. However, the definition presented in this review
is changed from the original definition in that paper to make it
cleaner in a categorical sense.

\section{Preliminaries}\label{Preliminaries}

\subsection{Topological Convexity Spaces}

\begin{definition}
A {\em topological convexity space} is a triple $(X,{\mathcal F},{\mathcal
  C})$, where $X$ is a set; ${\mathcal F}$ is a collection of subsets
of $X$ that is closed under finite unions and arbitrary intersections,
i.e. the collection of closed sets for some topology on $X$;
and ${\mathcal C}$ is a collection of subsets of $X$ that is closed
under directed unions and arbitrary intersections. Note that these
include empty unions and intersections, so $X$ and $\emptyset$ are in
both $\mathcal F$ and $\mathcal C$. Sets in $\mathcal F$ will be
called {\em closed} subsets of $X$ and sets in $\mathcal C$ will be
called {\em convex} subsets of $X$.
\end{definition}

The motivation here is that $(X,{\mathcal F})$ is a topological space,
while $(X,{\mathcal C})$ is an abstract convexity space. Abstract
convexity spaces are a generalisation of convex subsets of standard
Euclidean spaces. Abstract convexity spaces were defined in
\cite{Kay1971}, though in that paper, the definition did not require
$\mathcal C$ to be closed under nonempty directed unions. Closure
under directed unions was an additional property, called ``domain
finiteness''. Later authors incorporated closure under directed unions
into the definition of an abstract convexity space, and used the term
{\em preconvexity space} for a set with a chosen collection of subsets
that is closed under arbitrary intersections and contains the empty
set~\cite{Dawson1987}.

While the definition of an abstract convexity space captures many of
the important properties of convex sets in geometry, it also allows a
large number of interesting examples far beyond the original examples
from classical geometry, including many examples from combinatorics
and algebra. The resulting category of convexity spaces has many
natural closure properties~\cite{Dawson1987}. 

The definition above does not include any interaction between the
topological and convexity structures on $X$. While it will be
convenient to deal with such general spaces, it is also useful to
include compatibility axioms between the convexity and topological
structures. The following axioms from \cite{VanDeVel} are often used
to ensure suitable compatibility between topology and convexity
structure.

\begin{enumerate}[(i)]
  \item All convex sets are connected.
  \item All polytopes (convex closures of finite sets) are compact.
\item The hull operation is uniformly continuous relative to a metric which
generates the topology.
\end{enumerate}

We will modify the third condition to not require the topology to come
from a metric space, giving the weaker condition that the convex
closure operation preserves compact sets.

\begin{definition}\label{DefCompatible}
We will call a topological
convexity space {\em compatible} if it satisfies the two conditions  

\begin{enumerate}[(i)]
  \item All convex sets are connected.
  \item The convex closure of a (topologically) compact set is
    (topologically) compact.

\end{enumerate}
\end{definition}

At this point, we will introduce some notation for describing
topological convexity spaces. For any subset $A\subseteq X$, we will
write $[A]$ for the intersection of all convex sets containing $A$. To
simplify notation, when $A$ is finite, we will write
$[a_1,\ldots,a_n]$ instead of $[\{a_1,\ldots,a_n\}]$.

\begin{examples}\label{TCSexamples}
\ 
\begin{enumerate}
  
\item If $(X,d)$ is a metric space, then setting
${\mathcal
  F}$ to be the closed sets for the metric topology, i.e.
  $${\mathcal
  F}=\left\{A\subseteq X\middle|(\forall x\in X)\left(\bigwedge_{y\in
  A}d(x,y)=0\Rightarrow x\in A\right)\right\}$$ and $${\mathcal
  C}=\left\{A\subseteq X|(\forall x,y,z\in X)\left((x,z\in A\land
  d(x,z)=d(x,y)+d(y,z))\Rightarrow y\in A\right)\right\}$$
we have that $(X,{\mathcal F},{\mathcal C})$ is a 
topological convexity space. Finitely generated convex sets of $(X,d)$
are closed and compact, and the convex closure operation is uniformly
continuous. For convex sets to be connected, we need some suitable
interpolation property, namely that for any $r<d(x,y)$, there is some
$z\in [x,y]$ such that $d(x,z)=r$, since if this is not the case, then
$B(x,r)\cap[x,y]$ and $B(y,d(x,y)-r)\cap[x,y]$ form a disjoint open cover of
$[x,y]$. 

\item Let $L$ be a complete lattice. We define a topological convexity
  space structure by
  $${\mathcal F}=\left\{\bigcap_{i\in I}F_i\middle|(\forall i\in I)((\exists
    x_1,\ldots,x_{n_i}\in X)(F_i=\down\{x_1,\ldots,x_{n_i}\})\right\}$$
  and 
  $${\mathcal C}=\left\{I\subseteq X|(\forall x_1,x_2\in
  I)\left((\forall y\leqslant x_1)(y\in I)\land(x_1\lor x_2\in
  I)\right)\right\}$$ That is, ${\mathcal F}$ is the set of arbitrary
  intersections of finitely generated downsets (which are the closed
  sets for the weak topology \cite{Hoffman}) and ${\mathcal C}$ is the
  set of ideals of $L$. This topological convexity space is
  compatible. To prove connectedness of convex sets, we want to show
  that an ideal cannot be covered by two disjoint weak-closed
  sets. Suppose $U$ and $V$ are disjoint weak-closed sets that cover
  $I$. Let $a\in I\cap U$ and $b\in I\cap V$. Then $a\lor b\in I$, and
  if $a\lor b\in U$, then $b\in U$ contradicting disjointness of $U$
  and $V$. Similarly if $a\lor b\in V$ then $a\in V$.

\item Let $n\in {\mathbb Z}^+$ be a positive integer. Let $S\!_n$ be the
  group of permutations on $n$ elements. Let $\mathcal F$ consist of
  all subsets of $S\!_n$, and for any partial order $\preceq$ on $n$,
  let $P_\preceq=\{\sigma\in S\!_n| (\forall i,j\in
  \{1,\ldots,n\})(i\preceq j\Rightarrow \sigma(i)\leqslant\sigma(j))
  \}$, were $\leqslant$ is the usual total order on ${\mathbb
    Z}^+$. That is $P_\preceq$ is the set of permutations $\sigma$
  such that $\preceq$ is contained in $\sigma^{-1}(\leqslant)$.  let
  ${\mathcal C}=\{P_\preceq|\preceq\textrm{ is a partial order on }
  \{1,\ldots,n\}\}\cup\{\emptyset\}$. Since $S\!_n$ is finite, to prove
  that $(S\!_n,{\mathcal F},{\mathcal C})$ is a compatible convexity
  space, we just need to show that $\mathcal C$ is closed under
  intersection. This is straightforward. Since partial orders are
  closed under intersection, the poset of partial orders on
  $\{1,\ldots,n\}$, with a top element adjoined, is a lattice. Thus
  the intersection $P_\preceq\cap
  P_\sqsubseteq=P_{\preceq\lor\sqsubseteq}$, so $\mathcal C$ is closed
  under intersection. This is a metric topology, with the metric given
  by $d(\sigma,\tau)$ is the Cayley distance from $\sigma$ to $\tau$,
  under the Coxeter generators. That is, $d(\sigma,\tau)$ is the
  length of the shortest word equal to $\tau\sigma^{-1}$ in the
  generators $\{\tau_{i}|i=1,\ldots,n-1\}$, where
  $$\tau_i(j)=\left\{\begin{array}{ll} i+1&\textrm{if
  }j=i\\ i&\textrm{if }j=i+1\\ j&\textrm{otherwise}\end{array}\right.$$
  is the transposition of $i$ and $i+1$.
        
\item If $G$ is a topological group, or more generally a universal
  algebra equipped with a suitable topology, then we can define a
  topological convexity space by making subgroups (or more generally
  subalgebras) and the empty set convex, and keeping the closed sets
  from the topology.

\end{enumerate}
  
\end{examples}

Having defined the objects in the category of topological convexity
spaces, we need to define the morphisms.

\begin{definition}
  A {\em homomorphism} $f:(X,{\mathcal F},{\mathcal
    C})\morphism{}(X',{\mathcal F}',{\mathcal C}')$ between
  topological convexity spaces is a function $f:X\morphism{}X'$ such
  that for every $F\in{\mathcal F}'$, $f^{-1}(F)\in {\mathcal F}$ and
  for every $C\in{\mathcal C'}$, $f^{-1}(C)\in{\mathcal C}$.
\end{definition}

The condition that $f^{-1}(F)\in {\mathcal F}$ is the condition that
$f$ is continuous as a function between topological spaces. The
condition that $f^{-1}(C)\in{\mathcal C}$ is called {\em monotone}
by~\cite{Dawson1987}, by analogy with the example of endofunctions of
the real numbers. This was in the context of convexity spaces without
topological structure. Dawson~\cite{Dawson1987} uses the term
${\mathcal A}lign$ for the category of convexity spaces and monotone
homomorphisms, and ${\mathcal C}onvex$ for the category of convexity
spaces and functions whose forward image preserves convex
sets. However, this terminology has not been widely used, and later
authors have all considered the monotone homomorphisms as the natural
homomorphisms of abstract convexity spaces. In the case of topological
convexity spaces, the monotone condition is an even more natural
choice because it aligns well with the continuity condition and leads
to the Stone duality extension that we show in this paper. 

\begin{examples}
\ 
\begin{enumerate}
  
\item For the topological convexity space coming from a metric space,
  a homomorphism is a function $f:X\morphism{}Y$ such that whenever
  $d(x,z)=d(x,y)+d(y,z)$, we have
  $d(f(x),f(z))=d(f(x),f(y))+d(f(y),f(z))$. That is, $f$ embeds
  geodesics from $X$ into the geodesics in $Y$.

\item If $L$ and $M$ are complete lattices with the weak topology and
  convex sets are ideals, then topological convexity space
  homomorphisms from $L$ to $M$ are exactly sup-homomorphisms.

\end{enumerate}

\end{examples}

For the partial order convexity on $S\!_n$ from
Example~\ref{TCSexamples}.3, describing the topological convexity
space morphisms is more challenging.
We start by looking at half-spaces (convex sets with convex
complements). Half-spaces of $S\!_n$ are of the form
$C_{ij}=P_\preceq$, where $\preceq$ is the partial order where the
only non-trivial comparison is $i\preceq j$. That is,
$C_{ij}=\{\sigma\in S\!_n|\sigma(i)\leqslant \sigma(j)\}$.
We first consider
  automorphisms:

  \begin{lemma}\label{HalfSpaceCovers}
    If $i$, $j$, $k$ and $l$ are distinct, then the only half-spaces
    that contain $C_{ij}\cap C_{kl}$ are $C_{ij}$ and $C_{kl}$.
  \end{lemma}
  
  \begin{proof}
    For any half-space $C_{st}\not\in\{C_{ij},C_{kl}\}$, we need to
    find some $\sigma\in C_{ij}\cap C_{kl}$ with $\sigma\not\in
    C_{st}$. Suppose $s=j$ and $t\ne i$, then we can find a
    permutation $\sigma$ such that
    $\sigma(i)<\sigma(j)<\sigma(t)<\sigma(k)<\sigma(l)$. This $\sigma$
    is in $\C_{ij}\cap C_{kl}$, but not in $C_{st}$ as
    required. Similar permutations work for all combinations.
  \end{proof}

  \begin{lemma}\label{AutomorphismSn}
    An automorphism $f:(S\!_n,P(S\!_n),{\mathcal
      C})\morphism{}(S\!_n,P(S\!_n),{\mathcal C})$ is of the form
    $f(\sigma)=\theta\sigma\tau$ for some $\tau\in S\!_n$ and some
    $\theta\in\{e,\rho\}$ where $e$ is the identity permutation and
    $\rho$ is the permutation which reverses the order of all
    elements.
  \end{lemma}

  \begin{proof}
    It is easy to see that for $\tau\in S\!_n$, $f_\tau$ given by 
    $f_\tau(\sigma)=\sigma\tau$ is an automorphism of
    $(S\!_n,P(S\!_n),\mathcal C)$. Now we consider the stabiliser of the
    identity element. Since $\{C_{i(i+1)}|i=1,\ldots,(n-1)\}$ is the
    only set of $n-1$ half-spaces whose intersection contains only the
    identity permutation, any automorphism which fixes the identity
    permutation must fix this set. Furthermore, since $C_{(i-1)i}\cap
    C_{i(i+1)}\subseteq C_{(i-1)(i+1)}$, it follows that
    $f^{-1}(C_{(i-1)i})\cap f^{-1}(C_{i(i+1)})\subseteq
    f^{-1}(C_{(i-1)(i+1)})$. Since $f$ is an automorphism,
    $C_{(i-1)(i+1)}$ cannot be $C_{(i-1)i}$ or $C_{i(i+1)}$. By
    Lemma~\ref{HalfSpaceCovers}, it follows that $f^{-1}(C_{(i-1)i})$
    and $f^{-1}(C_{i(i+1)})$ are adjacent half-spaces. Since the set
    of half-spaces $\{C_{i(i+1)}|i=1,\ldots,(n-1)\}$ is permuted by
    $f^{-1}$, the only possible permutations are the identity and
    the reversal $C_{i(i+1)}\mapsto C_{(n-i)(n+1-i)}$. This reversal
    sends a permutation $\sigma$ to $\rho\sigma\rho$.

    We want to show that these are the only elements in the stabiliser
    of the identity. By applying $\rho\sigma\rho$ if necessary, we can
    change an element in the stabiliser of $e$ to one such that
    $f^{-1}$ fixes every $C_{i(i+1)}$. Now $C_{i(i+2)}$ is the unique
    half-space that contains $C_{i(i+1)}\cap C_{(i+1)(i+2)}$ that is
    not equal to either $C_{i(i+1)}$ or $C_{(i+1)(i+2)}$, so it is
    also fixed by $f^{-1}$. By induction, we can show that every
    $C_{ij}$ is fixed by $f^{-1}$, and thus $f$ is the identity.
    
  \end{proof}

  \begin{proposition}
    $f\!\!:S\!_n\morphism{}S\!_m$ is a surjective topological convexity space
    homomorphism, if and only if there is an injective function
    $g:m\morphism{}n$, such that $f$ is either given by
    \begin{enumerate}
      
\item    $f(\tau)(i)=|\{j\in\{1,\ldots,m\}|\tau(g(j))\leqslant
    \tau(g(i))\}|$. That is, $f(\tau)$ is the automorphism part of the
    automorphism---order-preserving-inclusion factorisation of $\tau g$.

    \hfil\xymatrix{n \ar[r]^{\tau} & n \\ m \ar@{ >->}[u]^g
      \ar[r]_{f(\tau)} & m \ar@{ >->}[u]_i}
    
    or

\item    $f(\tau)(i)=|\{j\in\{1,\ldots,m\}|\tau(g(j))\geqslant
    \tau(g(i))\}|$. That is, $f(\tau)$ is the automorphism part of the
    automorphism-order-preserving-inclusion factorisation of $\rho\tau
    g$, where $\rho$ is the order-reversing permutation on $n$.

    \hfil\xymatrix{n \ar[r]^{\tau} & n \ar[r]^{\rho} &n \\ m \ar@{ >->}[u]^g
      \ar[rr]_{f(\tau)} & & m \ar@{ >->}[u]_i}

    \end{enumerate}
    
    \end{proposition}

  \begin{proof}
    Firstly, we show that for an injective function $g:n\morphism{}m$,
    both the functions
    $$\alpha_g(\sigma)(i)=|\{j\in\{1,\ldots,m\}|\sigma(g(j))\leqslant
    \sigma(g(i))\}|$$
    and
        $$\delta_g(\sigma)(i)=|\{j\in\{1,\ldots,m\}|\sigma(g(j))\geqslant
    \sigma(g(i))\}|$$
    are surjective homomorphisms. We see that for any 
    $i,j\in\{1,\ldots,m\}$, $${\alpha_g}^{-1}(C_{ij})=\{\sigma\in S\!_n|\alpha_g(\sigma)(i)<\alpha_g(\sigma)(j)\}=\{\sigma\in S\!_n|\sigma(g(i))<\sigma(g(j))\}=C_{g(i)g(j)}$$
    and
    $${\delta_g}^{-1}(C_{ij})=\{\sigma\in S\!_n|\delta_g(\sigma)(i)<\delta_g(\sigma)(j)\}=\{\sigma\in S\!_n|\sigma(g(i))>\sigma(g(j))\}=C_{g(j)g(i)}$$
    so $\alpha_g$ and $\delta_g$ are homomorphisms.
For surjectivity, let $\phi\in S\!_m$. We need to show that
$\phi=\alpha_g(\tau)$ for some $\tau\in S_n$. Given the injections
$m\mmorphism{g}n$ and $\xymatrix@1{m \ar[r]^{\phi} & m \ar[r]^{i} &n}$
for any injective order-preserving $m\mmorphism{i}n$,
$\xymatrix@1{n & m \ar@{ >->}[l]_{g} \ar@{ >->}[r]^{i\phi} & n}$ is a
partial permutation of $n$, so it extends to a full permutation $\tau$
with $\alpha_g(\tau)=\phi$. Similarly, we have
$\delta_g(\rho\tau)=\alpha_g(\tau)=\phi$, so $\alpha_g$ and $\delta_g$
are both surjective.

    Conversely, let
$f:S\!_n\morphism{}S\!_m$ be a surjective homomorphism. Since $\{e\}$ is
convex, where $e$ is the identity homomorphism, $f^{-1}(\{e\})$ is
convex. Furthermore,
$f^{-1}(\{e\})=\bigcap_{i<j}f^{-1}(C_{ij})$. Since $f^{-1}$ preserves
convex sets, for every $i,j\in \{1,\ldots,m\}$ $f^{-1}(C_{ij})=C_{st}$
for some $s,t\in\{1,\ldots,n\}$. Furthermore, $f^{-1}(C_{ij}\cap
C_{jk})=C_{st}\cap C_{tu}$. Thus, we have
$f^{-1}(\{e\})=C_{i_1i_2\ldots i_m}=C_{i_1i_2}\cap
C_{i_2i_3}\cdots\cap C_{i_{m-1}i_m}$. If $f^{-1}(C_{12})=C_{i_1i_2}$,
then we can define $g(j)=i_j$, and we have that $f=\alpha_g$. If on
the other hand $f^{-1}(C_{12})=C_{i_{m-1}i_m}$, then we let
$g(j)=i_{m+1-j}$ and we have $f=\delta_g$.

  \end{proof}

  General homomorphisms are more difficult to describe.

\section{Preconvexity Spaces and the Adjunction with Topological
  Convexity Spaces}

\begin{definition}
  A {\em preconvexity space} (sometimes called a {\em closure space})
  is a pair $(X,{\mathcal P})$, where $X$ is a set and $\mathcal P$ is
  a collection of subsets of $X$ that is closed under arbitrary
  intersections and contains the empty set (since $X$ is an empty
  intersection, we also have $X\in{\mathcal P}$).
\end{definition}

This was~\cite{Kay1971}'s original definition of a convexity
space. However, later authors decided that closure under directed
unions should be a required property for a convexity space,
and~\cite{Dawson1987} introduced the term preconvexity space for these
spaces that do not require closure under directed unions.

\begin{definition}
A {\em homomorphism} $(X,{\mathcal P})\morphism{f}(X',{\mathcal P}')$
of preconvexity spaces is a function $X\morphism{f}X'$ such that for
any preconvex set $P\in{\mathcal P}'$, the inverse image
$f^{-1}(P)\in{\mathcal P}$. 
  \end{definition}

\begin{examples}\label{PreconvexityExample}

\item If $(X,{\mathcal F},{\mathcal C})$ is a topological convexity
  space, then $(X,{\mathcal F}\cap{\mathcal C})$ is a preconvexity
  space. Any topological convexity
  space homomorphism  $(X,{\mathcal F},{\mathcal
    C})\morphism{f}(X',{\mathcal F}',{\mathcal C}')$ is a preconvexity
  homomorphism. Conversely, if ${\mathcal C}'$ consists of directed
  unions from ${\mathcal F}'\cap{\mathcal C}'$, and ${\mathcal F}'$
  consists of intersections of finite unions from ${\mathcal
    F}'\cap{\mathcal C}'$, then any preconvexity homomorphism 
$(X,{\mathcal F}\cap{\mathcal
    C})\morphism{g}(X',{\mathcal F}'\cap{\mathcal C}')$ is a
  topological convexity homomorphism. 
  
\end{examples}


Example~\ref{PreconvexityExample} gives a functor
$\TC\morphism{CC}\Preconvex$ that sends every topological convexity
space to the preconvexity space of closed convex sets.  The action on
morphisms simply reinterprets the topological convexity homomorphism
as a preconvexity homomorphism.

This functor is not a fibration. Indeed some morphisms in
\Preconvex\ do not even have lifts. Consider for example
$$\left(\{0,1,2\},\big\{\emptyset,\{0\},\{1\},\{2\},\{0,1,2\}\big\}\right)\morphism{f}\left(\{0,1\},\big\{\emptyset,\{0,1\}\big\}\right)$$
given by $f(0)=f(1)=0$ and $f(2)=1$. This has no lift to a topological
convexity space homomorphism with codomain
$$\left(\{0,1\},\big\{\emptyset,\{0,1\}\big\},\big\{\emptyset,\{0\},\{0,1\}\big\}\right)$$
because the preimage of $\{0\}$ under $f$ is $\{0,1\}$, so for this to
be a homomorphism, we would need $\{0,1\}$ to be convex. However,
since $\{0\}$ and $\{1\}$ are both closed, it follows that $\{0,1\}$
must be closed, so the resulting space is no longer in the fibre above $$\left(\{0,1,2\},\big\{\emptyset,\{0\},\{1\},\{2\},\{0,1,2\}\big\}\right)$$

The closed-convex functor does have a right adjoint, $IS$, which sends
the preconvexity space $(X,{\mathcal P})$ to $(X,\overline{\mathcal
  P},\widetilde{\mathcal P})$ where $\widetilde{\mathcal P }$ is the
closure of $\mathcal P$ under directed unions, and
$\overline{\mathcal P }$ is the closure of $\mathcal P$ under finite
unions and arbitrary intersections.

\begin{lemma}\label{ConvexSetsFromPreconvexity}
  For any preconvexity space $(X,{\mathcal P})$, the set
  $\widetilde{\mathcal P}$ is closed under directed unions and
  arbitrary intersections.
\end{lemma}

\begin{proof}
By definition, $\widetilde{\mathcal P}$ is closed under directed
unions, so we just need to show that it is closed under intersections.
Let $\{P_i|i\in I\}$ be a family of elements of $\widetilde{\mathcal
  P}$. By definition, for every $i\in I$, there is a directed
${\mathcal D}_i\subseteq{\mathcal P}$ with $P_i=\bigcup {\mathcal
  D}_i$. W.l.o.g. assume every ${\mathcal D}_i$ is down-closed in
$\mathcal P$. We will show that
\begin{equation}
\bigcap_{i\in
  I}P_i=\bigcup_{\stackrel{f:I\morphism{}{\mathcal P}}{(\forall i\in
    I)f(i)\in{\mathcal D}_i}}\bigcap_{i\in I} f(i)\label{IntersectionDirected} 
\end{equation}
That is, the intersection of the family $\{P_i|i\in I\}$ is the union
over all choice functions $f$, of the intersection of $\{f(i)|i\in
I\}$. Every $f(i)\in{\mathcal P}$, so this intersection $\bigcap_{i\in I}
f(i)$ is also in $\mathcal P$, and the set of choice functions is
directed, since every ${\mathcal D}_i$ is directed and down-closed, so
for choice functions $f,g:I\morphism{}{\mathcal P}$ the join $(f\lor
g)(i)=f(i)\cup g(i)$ is also a choice function.
Equation~\eqref{IntersectionDirected} therefore shows that $\bigcap_{i\in
  I}P_i\in\widetilde{\mathcal P}$.

To prove Equation~\eqref{IntersectionDirected}, first let $x\in \bigcap_{i\in
  I}P_i$. Since $(\forall i)(x\in P_i)$, and $P_i=\bigcup{\mathcal
  D_i}$, there is some $D_{i,x}\in{\mathcal D}_i$ with $x\in
D_{i,x}$. Thus, we can take the choice function $f_x(i)=D_{i,x}$, and
deduce $x\in \bigcap_{i\in I} f_x(i)$. Conversely, let $$x\in \bigcup_{\stackrel{f:I\morphism{}{\mathcal P}}{(\forall i\in
    I)f(i)\in{\mathcal D}_i}}\bigcap_{i\in I} f(i)$$
There must be some choice function $f$ with $x\in \bigcap_{i\in I}
f(i)$. Since $f(i)\in {\mathcal D}_i$, it follows that $f(i)\subseteq
P_i$, so $x\in P_i$ for every $i\in I$. Thus $x\in \bigcap_{i\in I}P_i$.
\end{proof}

\begin{remark}
The proof of Lemma~\ref{ConvexSetsFromPreconvexity} does not actually
require the axiom of choice, because there are canonical choices for
all choice functions needed --- for each $P_i$, we can let ${\mathcal
  D}_i=\{P\in{\mathcal P}|P\subseteq P_i\}$, and since every ${\mathcal
  D}_i$ is a downset, we can set $D_{i,x}=\overline{\{x\}}$ for every
$i\in I$, where $\overline{\{x\}}$ is the convex-closed closure of
$\{x\}$.
\end{remark}

\begin{lemma}
Every $F\in\overline{\mathcal P}$ is of the form $\bigcap{\mathcal
  F}$, where $${\mathcal F}\subseteq\{P_1\cup\cdots\cup
P_n|P_1,\ldots,P_n\in {\mathcal P}\}$$
\end{lemma}

\begin{proof}
  Let $\widehat{\mathcal P}=\{P_1\cup\cdots\cup
P_n|P_1,\ldots,P_n\in {\mathcal P}\}$ be the set of finite unions from
$\mathcal P$. We need to show that the set $\left\{\bigcap {\mathcal
  F}|{\mathcal F}\subseteq \widehat{\mathcal P}\right\}$ is closed
under finite unions. (By definition, it is closed under
arbitrary intersections.) Let $F_1=\bigcap {\mathcal F}_1$ and
$F_2=\bigcap {\mathcal F}_2$ for ${\mathcal F}_1,{\mathcal
  F}_2\subseteq\widehat{\mathcal P}$. Let ${\mathcal F}_{12}=\{P_1\cup
P_2|P_1\in{\mathcal F}_1,P_2\in{\mathcal F}_2\}$. We will show that
$F_1\cup F_2=\bigcap {\mathcal F}_{12}$. Clearly, for every
$P_1\in{\mathcal F}_1$, and $P_2\in{\mathcal F}_2$, we have
$F_1\subseteq P_1$ and $F_2\subseteq P_2$, so $F_1\cup F_2\subseteq
P_1\cup P_2$. Conversely, suppose $x\not\in F_1\cup F_2$. Then there
is some $P_1\in{\mathcal F}_1$ and some $P_2\in{\mathcal F}_2$ with
$x\not\in P_1$ and $x\not\in P_2$. It follows that $x\not\in P_1\cup
P_2\in{\mathcal F}_{12}$, so $x\not\in\bigcap{\mathcal F}_{12}$. 
\end{proof}

\begin{lemma}
 \
  \begin{enumerate}
    \item For a set $X$, the identity function on $X$ is a
      preconvexity homomorphism $(X,{\mathcal
        P})\morphism{}(X,{\mathcal P}')$ if and only if ${\mathcal
        P'}\subseteq {\mathcal P}$.

    \item For a set $X$, the identity function on $X$ is a
      topological convexity homomorphism $(X,{\mathcal
        F},{\mathcal C})\morphism{}(X,{\mathcal F}',{\mathcal C}')$ if and only if ${\mathcal
        F'}\subseteq {\mathcal F}$ and ${\mathcal
        C'}\subseteq{\mathcal C}$.

  \end{enumerate}
      
  \end{lemma}

\begin{proof}
This is immediate from the definition.
\end{proof}

\begin{proposition}\label{AdjointTCPreconvex}
  The assignment $IS$ that sends the preconvexity space $(X,{\mathcal
    P})$ to the topological convexity space $(X,\overline{\mathcal
    P},\widetilde{\mathcal P})$ and the preconvexity homomorphism
  $(X,{\mathcal P})\morphism{f}(X',{\mathcal P}')$ to $f$ considered
  as a topological convexity homomorphism, is a functor, and is right
  adjoint to the functor $CC:\TC\morphism{}\Preconvex$.
\end{proposition}

\begin{proof}
Because the forgetful functor to \Set\ sends $IS$ to the identity
functor, the functoriality of $IS$ is automatic provided it is
well-defined. That is, if any preconvexity homomorphism $(X,{\mathcal
  P})\morphism{f}(X',{\mathcal P}')$ is a topological convexity
homomorphism from $(X,\overline{\mathcal
    P},\widetilde{\mathcal P})$ to $(X',\overline{\mathcal
    P'},\widetilde{\mathcal P'})$. For the adjunction, we need to
demonstrate that for any topological convexity space $(X,{\mathcal
  F},{\mathcal C})$ and any preconvexity space $(X',{\mathcal P}')$, a
function $f:X\morphism{}X'$ is a topological convexity space
homomorphism $(X,{\mathcal
  F},{\mathcal C})\morphism{f}(X',\overline{\mathcal
    P'},\widetilde{\mathcal P'})$ if and only if it is a preconvexity
homomorphism $(X,{\mathcal
  F}\cap{\mathcal C})\morphism{f}(X',{\mathcal P'})$. The ``only if'' part
is obvious.

Suppose $(X,{\mathcal F}\cap{\mathcal C})\morphism{f}(X',{\mathcal
  P'})$ is a preconvexity homomorphism. Let $F\in\overline{\mathcal
  P'}$. We want to show that $f^{-1}(F)\in{\mathcal F}$. Now
$F\in\overline{\mathcal P'}$ means $F=\bigcap{\mathcal U}$ where
${\mathcal U}\subseteq \widehat{\mathcal P'}$. Now if
$P_1\cup\cdots\cup P_n\in\widehat{\mathcal P'}$, then
$f^{-1}(P_1\cup\cdots\cup P_n)=f^{-1}(P_1)\cup\cdots\cup f^{-1}(P_n)$
is a finite union of sets from ${\mathcal F}\cap{\mathcal C}$, so
since ${\mathcal F}$ is closed under finite unions,
$f^{-1}(P_1\cup\cdots\cup P_n)\in{\mathcal F}$. Therefore
$f^{-1}(F)=\bigcap\{f^{-1}U|U\in{\mathcal U}\}$ and
$\{f^{-1}U|U\in{\mathcal U}\}\subseteq \mathcal F$, so as $\mathcal F$
is closed under arbitrary intersections, $f^{-1}(F)\in{\mathcal
  F}$. Similarly, let $C=\bigcup {\mathcal D}$, where ${\mathcal
  D}\subseteq {\mathcal P}'$ is a directed downset. For every
$D\in{\mathcal D}$, we have $f^{-1}(D)\in\mathcal C$, and for any
$D_1,D_2\in\mathcal D$, there is some $D_{12}\in\mathcal D$ with
$D_1\subseteq D_{12}$ and $D_2\subseteq D_{12}$. It follows that
$f^{-1}(D_1)\subseteq f^{-1}(D_{12})$ and $f^{-1}(D_2)\subseteq
f^{-1}(D_{12})$. Therefore, $\{f^{-1}(D)|D\in{\mathcal D}\}$ is
directed. Now $f^{-1}(C)=f^{-1}(\bigcup {\mathcal D})=\bigcup
\{f^{-1}(D)|D\in{\mathcal D}\}$. Since $\{f^{-1}(D)|D\in{\mathcal
  D}\}\subseteq \mathcal C$, and $\mathcal C$ is closed under directed
unions, it follows that $f^{-1}(C)\in{\mathcal C}$. Thus $f$ is a
homomorphism of topological convexity spaces.

Well-definedness of the functor $IS$ also follows from the
adjunction, because ${\mathcal P}\subseteq\overline{\mathcal P}\cap\widetilde{\mathcal
  P}$, so the identity function on $X$ is always a preconvexity
homomorphism $(X,\overline{\mathcal P}\cap\widetilde{\mathcal
  P})\morphism{i}(X,{\mathcal P})$. Thus the composite
$$(X,\overline{\mathcal P}\cap\widetilde{\mathcal
  P})\morphism{i}(X,{\mathcal P})\morphism{f}(X',{\mathcal P}')$$
is a preconvexity homomorphism, so by the adjunction, it is a
topological convexity space homomorphism 
$(X,\overline{\mathcal P},\widetilde{\mathcal
  P})\morphism{f}(X',\overline{{\mathcal P}'},\widetilde{{\mathcal P}'})$

\end{proof}

\begin{corollary}\label{IdempotentTCPreconvex}
The adjunction $CC\dashv IS$ is idempotent.
\end{corollary}

\begin{proof}
The counit and unit of the adjunction are both the identity function
viewed as a homomorphism in the relevant category. The triangle
identities for the adjunction therefore give an isomorphism of spaces,
showing that the adjunction is idempotent.
\end{proof}

For an idempotent adjunction, a natural question is what are the fixed
points?

\begin{proposition}\label{Teetotal}

A topological convexity space $X=(X,{\mathcal F},{\mathcal C})$
satisfies $IS\circ CC(X)=X$ if and only if $X$ satisfies the conditions:
  
\begin{enumerate}
  
\item\label{ConvexSets} Every convex set is a directed union of closed convex sets.

\item\label{ClosedSets} For every $V\in\mathcal F$ and any $x\in
  X\setminus V$, there are sets $C_1,\ldots,C_n\in{\mathcal
    F}\cap{\mathcal C}$ such that $V\subseteq C_1\cup\ldots\cup\C_n$
  and $x\not\in C_1\cup\ldots\cup\C_n$.

\end{enumerate}

\end{proposition}

\begin{proof}
  The counit of the adjunction is the identity function on the
  underlying sets. Thus $\overline{\left({\mathcal F}\cap{\mathcal
      C}\right)}\subseteq {\mathcal F}$ and
  $\widetilde{\left({\mathcal F}\cap{\mathcal C}\right)}\subseteq
  {\mathcal C}$. Let $A\in{\mathcal C}$ be convex in $X$. By
  Condition~1, $A$ is a directed union of sets in ${\mathcal
    F}\cap{\mathcal C}$. By definition, this is in
  $\widetilde{\left({\mathcal F}\cap{\mathcal C}\right)}$.

  Now let $V\in{\mathcal F}$. For any $W=C_1\cup\cdots\cup C_n$ with
  $C_i\in{\mathcal F}\cap{\mathcal C}$, $W\in\overline{\left({\mathcal
      F}\cap{\mathcal C}\right)}$ by definition. Thus, by Condition~2,
  for every $x\in X\setminus V$, there is some
  $W\in\overline{\left({\mathcal F}\cap{\mathcal C}\right)}$ with
  $V\subseteq W$ and $x\not\in W$. Now, clearly $V$ is the
  intersection of all these $W$ for all $x\not\in V$. Since
  $\overline{\left({\mathcal F}\cap{\mathcal C}\right)}$ is closed
  under arbitrary intersections, this implies $V\in
  \overline{\left({\mathcal F}\cap{\mathcal C}\right)}$.

  Conversely, if $X$ is a fixed point of the adjunction, i.e. $IS\circ
  CC(X)=X$, then ${\mathcal C}=\widetilde{(F\cap C)}$, which is exactly
  Condition~1. Also ${\mathcal F}=\overline{(F\cap C)}$, meaning that
  for every $V\in{\mathcal F}$, we have $V=\bigcap {\mathcal U}$ where
  $\mathcal U$ is a family of finite unions of sets from $\mathcal
  F\cap\mathcal C$. Since $V=\bigcap {\mathcal U}$, for any $x\not\in
  V$, there is some $U\in{\mathcal U}$ with $x\not\in U$. By
  definition, $U=C_1\cup\cdots\cup C_n$ for some $C_1,\ldots,C_n\in
  {\mathcal F} \cap{\mathcal C}$, which is Condition~2. 
\end{proof}

We will call a topological convexity space {\em teetotal} if the
conditions of Proposition~\ref{Teetotal} hold. The teetotal conditions
are closely related to the compatible conditions from
Definition~\ref{DefCompatible}. However, there are compatible spaces
which are not teetotal.

\begin{example}
  $l^2$ is the vector-space of square-summable sequences of real
  numbers, with the $l^2$ norm. Since $l^2$ is a metric space, it is
  easy to check that it is a compatible topological convexity space.

  Let $F$ be the unit sphere, which is a closed set, and let $x=0$. In
  order for $l^2$ to be teetotal, we need to find a finite family of
  closed convex subsets $C_1,\ldots,C_n$ such that $F\subseteq
  C_1\cup\cdots\cup C_n$ and $x\not\in C_1\cup\cdots\cup C_n$.  For
  closed convex $C_i$ and $x\not\in C_i$, since $C_i$ is closed, there
  is an open ball containing $x$ disjoint from $C_i$. Let
  $d=\sup\{r\in{\mathbb R}|B(x,r)\cap C_i=\emptyset\}$ be the distance
  from $x$ to $C_i$. Since $B(x,d)$ is the directed union of
  $\{B(x,r)|r<d\}$, it follows that $B(x,d)\cap C_i=\emptyset$.

  We first show that if $C$ is a closed convex set that does not
  contain 0, then there is a unique $x\in C$ that minimises $\lVert
  x\rVert$. If there is no $x\in C$ that minimises $\lVert x\rVert$,
  then there must be a sequence $a_1,a_2,\ldots \in C$ such that
  $\lVert a_i\rVert$ is strictly decreasing and
  $$\lim_{n\rightarrow\infty}\lVert a_n\rVert=\inf_{y\in C}\lVert
  y\rVert$$
  Since $[a_1,\ldots,a_n]$ is compact for every $n$, there is a point
  $b_n\in[a_1,\ldots,a_n]$ that minimises $\lVert b\rVert$. In
  particular, this means that for any $i<n$ and any $\epsilon>0$,
  $\lVert b_n+\epsilon(b_i-b_n)\rVert\geqslant\lVert b_n\rVert$. Squaring both
  sides gives $$2\epsilon\langle b_i,b_n\rangle-2\epsilon\langle b_n,b_n\rangle+
\epsilon^2\langle b_i-b_n,b_i-b_n\rangle>0$$
Taking the limit as $\epsilon\rightarrow 0$ gives $\langle
b_i,b_n\rangle>\langle b_n,b_n\rangle$. Thus
\begin{align*}
  \lVert b_i-b_n\rVert^2&=\lVert b_i\rVert^2+\lVert
  b_n\rVert^2-2\langle b_i,b_n\rangle\\
  &\leqslant \lVert b_i\rVert^2-\lVert
  b_n\rVert^2
\end{align*}
Since $\lVert b_n\rVert^2$ is a decreasing sequence, bounded below by
0, it converges to some limit $r$. Thus $\lVert
b_i-b_n\rVert^2\leqslant \lVert b_i\rVert^2-r$ for any $i<n$. Thus
$b_n$ is a Cauchy sequence, so it converges to some limit
$b_\infty$. Now since $C$ is closed, $b_\infty\in C$, and
$$\lVert b_\infty\rVert=\lim_{n\rightarrow\infty} \lVert
b_n\rVert=\inf_{y\in C}\lVert y\rVert$$
Thus $b_\infty$ is a nearest point in $C$ to $0$. If $x$ is another
point with minimal norm, then $\frac{x+b_\infty}{2}$ must have smaller
norm. Thus $b_\infty$ is the unique point with smallest norm.

Now for any $y\in C$, since $C$ is convex, we have that $\lVert
b_\infty+\epsilon(y-b_\infty)\rVert>\lVert b_\infty\rVert$, and by the
above argument, $\langle y,b_\infty \rangle\geqslant \langle
b_\infty,b_\infty \rangle$. Thus $C\subseteq \{x\in l^2|\langle
x,b_\infty\rangle>\frac{1}{2}\lVert b_\infty\rVert^2\}$. That is, every closed
convex set is contained in a half-space. 

We can therefore find half-spaces $H_1,\ldots,H_n$ with $x\not\in H_i$
and $C_i\subseteq H_i$. Thus, we may assume that $F\subseteq
H_1\cup\cdots\cup H_n$. Half-spaces that do not contain the origin are
sets of the form $H_{w,a}=\{v\in l^2|\langle v,w\rangle>a\}$ for some
$w\in l^2$ and $a\in\mathbb R^+$. Given a finite family $H_1,\ldots,
H_n=H_{w_1,a_1},\ldots,H_{w_n,a_n}$, we can find a unit vector $w$
that is orthogonal to all of $w_1,\ldots,w_n$. This means that
$w\not\in H_i$ for all $i$, and $w\in F$, contradicting the assumption
that $F\subseteq H_1\cup\cdots\cup H_n$. Therefore, $l^2$ does not
satisfy the teetotal axioms.

  The teetotal interior $IS\circ CC(l^2)$ has the same convex sets, but closed
  sets are intersections of finite unions of closed half-spaces. We
  can check that this is the product topology on $l^2$ as a real
  vector space.
  
\end{example}

\begin{example}

  Let $(X,d)$ be a metric space, where $X=\bigcup_{n\in{\mathbb
      N}}[n]^n$ is the set of finite lists with entries bounded by
  list length. The distance is given by $d(u,v)=l(u)+l(v)-l(u\cap
  v)$, where $l(u)$ is the length of the list $u$ and $u\cap v$ is the
  longest list which is an initial sublist of both $u$ and $v$. The
  induced topology is clearly discrete. The complement of the empty
  list is not contained in a finite union of convex subsets that does
  not contain the empty list. In particular, a convex subset of $X$
  that does not contain $\emptyset$ must consist of lists that all
  start with the same first element. Since there are infinitely many
  possible first elements, a finite collection of convex sets that do
  not contain the empty list cannot cover $X\setminus\emptyset$.

  The space $(X,d)$ is a metric space and every closed ball is
  compact. However, it is not a fixed point of the adjunction between
  $\TC$ and $\Preconvex$.

\end{example}

For a compatible topological space to be teetotal, an additional
property is needed.

\begin{proposition}\label{GeneralTeetotal}
  If $(X,{\mathcal F},{\mathcal C})$ is a compatible topological convexity space
  with the following properties:

  \begin{itemize}

    \item There is a basis of open sets that are convex, whose closure
      is convex and compact.

    \item $(X,{\mathcal F})$ is Hausdorff.
      

      
      
    \item If $A$ is closed convex and $x\not\in A$, then there is a
      closed convex set $H$ such that $H^c$ is convex, with
      $A\subseteq H$ and $x\not\in H$. (This property, without the
      topological constraints, is often used in the literature, where
      it is called the Kakutani condition.)
    
  \end{itemize}

  then $(X,{\mathcal F},{\mathcal C})$ is fixed by the adjunction.
\end{proposition}

\begin{proof}
We need to show that for any closed $C\in{\mathcal F}$, and any
$x\not\in C$, there is a finite set of closed convex sets whose union
covers $C$ but does not contain $x$. Let $U$ be an open subset of
$C^c$, containing $x$ such that $U$ is convex and $\overline{U}$ is
convex and compact. Let $A=\overline{U}\setminus U$. For any $a\in A$,
by the Hausdorff property, we can find an open $U_a$ that contains
$a$, whose closure does not contain $x$. Since convex open sets with
convex closure form a basis of open sets, we can find a convex open
$U'_a$ with convex closure that does not contain $x$. Since $A$ is
compact, it is covered by a finite subset $U'_{a_1}\cup\cdots\cup
U'_{a_n}$. Now each $\overline{U'_{a_i}}$ is contained in a closed
convex $H_{a_i}$ which does not contain $x$.

For any $y\in C$, since $[x,y]$ is connected (by compatibility), it
cannot be the union $([x,y]\cap
U)\cup\left([x,y]\cap\overline{U}^c\right)$, so $[x,y]\cap
A\ne\emptyset$. Let $z\in [x,y]\cap A$. Since $H_{a_i}$ cover $A$, we
have $z\in H_{a_i}$ for some $i$. Now if $y\in {H_{a_i}}^c$, then
since ${H_{a_i}}^c$ is convex and contains $x$, it follows that $z\in
{H_{a_i}}^c$ contradicting $z\in {H_{a_i}}$. Thus, we must have $y\in
{H_{a_i}}$. Since $y\in C$ is arbitrary, we have that $C\subseteq
H_{a_1}\cup\cdots\cup H_{a_n}$ as required.

We also need to show that every convex set is a directed union of
closed convex sets. Let $C\in{\mathcal C}$ be a convex set.
Let ${\mathcal D}=\{[F]|F\subseteq C, F\textrm{ finite}\}$ be the
collection of finitely generated convex subsets of $C$. Since finite
sets are closed under binary unions, $\mathcal D$ is directed. Since
the convex closure of any finite set is closed, it follows that $C$ is
a directed union of closed convex sets as required.
\end{proof}

For a metric space, these conditions can be simplified to give more
natural conditions.

\begin{lemma}\label{KakutaniMetric}
  If $X$ is a topological convexity 
  space where intervals are closed, satisfying the Kakutani property that
  every pair of disjoint closed convex sets are separated by a closed
  half-space, then for any $x,s,t,p,q,r\in X$ with $s\in [x,p]$, $t\in
  [x,q]$ and $r\in [p,q]$, we have $[x,r]\cap[s,t]\ne\emptyset$.
\end{lemma}

\begin{proof}
  If $[x,r]\cap[s,t]=\emptyset$, then $[x,r]$ and $[s,t]$ are
  disjoint closed convex sets, so by the Kakutani propery, there is a
  closed half-space $H$ such that $[x,r]\subseteq H$ and
  $[s,t]\subseteq H^c$. Now if $p\in H$, then since $x\in H$ and $H$
  is convex, we get $s\in H$, contradicting $[s,t]\subseteq H^c$. This
  is a contradiction, so we must have $p\in H^c$. A similar argument
  shows that $q\in H^c$. However, since $H^c$ is convex, it follows
  that $r\in H^c$, contradicting $[x,r]\subseteq H$. This
  contradiction disproves  $[x,r]\cap[s,t]=\emptyset$, so we must have
  $[x,r]\cap[s,t]\ne\emptyset$
\end{proof}

\begin{lemma}\label{PolytopeBetween}
If $(X,d)$ is a metric space, such that every closed ball is compact,
every open ball is convex, every convex set is connected, every pair
of disjoint closed convex sets are separated by a closed half-space (a
closed convex set with convex complement), and every interval $[a,b]$
is isomorphic (as a topological convexity space) to the real interval
$[0,1]$ then for any convex compact $A\subseteq X$ and any $x\in X$,
we have
  $$[x,A]=\bigcup\{[x,y]|y\in A\}$$
\end{lemma}

\begin{proof}
We need to show that $\bigcup\{[x,y]|y\in A\}$ is closed under the
betweenness relation. Let $s,t\in \bigcup\{[x,y]|y\in A\}$, and let
$z\in[s,t]$. Let $s\in [x,p]$ and $y\in[s,q]$ for $p,q\in A$. We will
show that $z\in [x,r]$ for some $r\in [p,q]$. Since $[s,t]\cong[0,1]$,
we have that $[s,t]=[s,z]\cup[z,t]$. 
For $r\in [p,q]$, if
$[x,r]\cap [s,z]\ne\emptyset$ and $[x,r]\cap [z,t]\ne\emptyset$, then
clearly $z\in[x,r]$. Thus if $(\forall r\in [p,q])(z\not\in[x,r])$,
then $(\forall r\in
[p,q])(([x,r]\cap[s,z]=\emptyset)\lor([x,r]\cap[z,t]=\emptyset))$, so $$[p,q]=\{r\in[p,q]|[x,r]\cap[s,z]=\emptyset\}\cup\{r\in[p,q]|[x,r]\cap[z,t]=\emptyset\}$$
and this union is disjoint. By connectedness of $[p,q]$, we just need
to show that $\{r\in X|[x,r]\cap[s,z]=\emptyset\}$ and $\{r\in
X|[x,r]\cap[z,t]=\emptyset\}$ are open to reach a contradiction, which
would prove $z\in[x,r]$ for some $r\in[p,q]$. Let $U=\{r\in
X|[x,r]\cap[s,z]=\emptyset\}$, and let $v\in U$. We want to show that
there is some $\epsilon$ such that $B(v,\epsilon)\subseteq U$. Now 
$[s,z]\cap[x,r]=\emptyset$, which means $(\forall
w\in[s,z])(d(x,w)+d(w,r)\ne d(x,r))$. Since $[s,z]$ is compact, the
function $f(w)=d(x,w)+d(w,r)-d(x,r)$ is bounded away from zero on
$[s,z]$. Let $\delta$ be a lower bound. Now if $v'\in
B\left(v,\frac{\delta}{2}\right)$, then for any $w\in[s,z]$, we have
\begin{align*}
  d(x,w)+d(w,v')&\geqslant d(x,w)+d(w,v)-d(v,v')\\
  &> d(x,v)+\delta-\frac{\delta}{2}\\
  &\geqslant d(x,v')-d(v',v)+\frac{\delta}{2}\\
  &> d(x,v')
\end{align*}
so $w\not\in [x,v']$, i.e. $v'\in U$. Thus
$B\left(v,\frac{\delta}{2}\right)\subseteq U$, meaning $U$ is open as required.
\end{proof}

\begin{corollary}\label{MetricTeetotal}
  If $(X,d)$ is a metric space, such that every closed ball is
  compact, every open ball is convex, every convex set is connected,
  every pair of disjoint closed convex sets are separated by a closed
  half-space, and every interval $[a,b]$ is isomorphic to the real
  interval $[0,1]$ then the induced topological convexity space is
  fixed by the adjunction.
\end{corollary}

\begin{proof}
We will show that the conditions of Proposition~\ref{GeneralTeetotal}
hold in this case. The Hausdorff condition is always true for metric
spaces.

Next, we need to show that the convex closure of a finite set is
compact. We will do this inductively. By Lemma~\ref{PolytopeBetween},
we have that
$[x_1,\ldots,x_n]=\bigcup\{[x_1,y]|y\in[x_2,\ldots,x_n]\}$. By the
induction hypothesis $[x_2,\ldots,x_n]$ is compact. This means that
$[x_2,\ldots,x_n]\subseteq B(x_1,r)$ for some $r\in{\mathbb R}^+$. It
follows that $[x_1,\ldots,x_n]\subseteq B(x_1,r)$, since $B(x_1,r)$ is
convex. Therefore, it is sufficient to prove that 
$[x_1,\ldots,x_n]$ is closed. Let $z\not\in[x_1,\ldots,x_n]$. We want
to prove that there is some open ball about $z$ that is not contained
in $[x_1,\ldots,x_n]$. For any $y\in[x_2,\ldots,x_n]$, we know
$z\not\in[x,y]$, so $d(x,z)+d(z,y)-d(x,y)> 0$. For $y\in
[x_2,\ldots,x_n]$, let $f(y)=d(x,z)+d(z,y)-d(x,y)$. Then $f(y)$ is a
continuous function $[x_2,\ldots,x_n]\rightarrow{\mathbb R}$. Since
$[x_2,\ldots,x_n]$ is compact, $f$ attains its lower bound, so in
particular, there is some $\epsilon>0$ such that $f(y)>\epsilon$ for
all $y\in [x_2,\ldots,x_n]$. Now if $d(z,z')<\frac{\epsilon}{2}$, then
for any $y\in[x_2,\ldots,x_n]$, $d(x,z')+d(z',y)>d(x,z)-\frac{\epsilon}{2}+d(z,y)-\frac{\epsilon}{2}>d(x,y)$,
so $z'\not\in [x_1,y]$. It follows that $z'\not\in [x_1,\ldots,x_n]$,
so $[x_1,\ldots,x_n]$ is closed, as required.
\end{proof}

In the other direction, it is natural to ask which preconvexity spaces
are fixed by the monad $CC\circ IS$.  The functor $CC\circ IS$ sends a
preconvexity space, $(X,{\mathcal P})$ to the space
$(X,\overline{\mathcal P}\cap\widetilde{\mathcal P})$. We will call a
preconvexity space $(X,{\mathcal P})$ {\em geometric} if
$\overline{\mathcal P}\cap\widetilde{\mathcal P}={\mathcal P}$.

\begin{proposition}
  If $X$ is finite, then any preconvexity space $(X,\mathcal P)$ is
  geometric.
\end{proposition}

\begin{proof}
If $X$ is finite, then $\widetilde {\mathcal P}={\mathcal P}$, so
$\widetilde{{\mathcal P}}\cap\overline{{\mathcal P}}={\mathcal P}$ as required.
\end{proof}

A natural question is whether this extends to topologically discrete
spaces. In fact, there are preconvexity spaces where all sets are in
both $\overline{\mathcal P}$ and $\widetilde{\mathcal P}$, but not in
$\mathcal P$.


\begin{example}
  Let $X=\mathbb N$. Let $\mathcal P$ consist of all subsets of
  $\mathbb N$ whose complement is infinite or empty. Clearly every
  subset of $\mathbb N$ is a finite union from $\mathcal P$, and also
  a directed union from $\mathcal P$. Thus $(X,\mathcal P)$ is a
  non-geometric example where all sets are closed and all sets are
  convex.
\end{example}

\begin{proposition}
  Every $T_0$ preconvexity space (meaning for any $x\ne y$, there is a
  preconvex set containing exactly one of $x$ and $y$) embeds in a
  geometric preconvexity space.
\end{proposition}

\begin{proof}
For a $T_0$ preconvexity space $(X,{\mathcal P})$, let $Y={\mathcal P}$ and
${\mathcal Q}=\big\{\{S\in {\mathcal P}|S\subseteq R\}| R\in{\mathcal
  P}\big\}$. Now the inclusion $X\morphism{i}Y$ given by
$i(x)=\bigcap\{P\in{\mathcal P}|x\in P\}$, is an embedding of
preconvexity spaces, meaning that for $A\subseteq X$, we have
$A\in{\mathcal P}$ if and only if $A=i^{-1}(B)$ for some
$B\in{\mathcal Q}$. Clearly if $A\in{\mathcal P}$, then
$\{S\in{\mathcal P}|S\subseteq A\}\in {\mathcal Q}$. Now it is easy to
see that $a\in i^{-1}\big(\{S\in{\mathcal P}|S\subseteq A\}\big)$
if and only if $i(a)\subseteq A$, if and only if $a\in A$. Thus $A=
i^{-1}\big(\{S\in{\mathcal P}|S\subseteq A\}\big)$. Conversely, let
${\mathcal R}\in{\mathcal Q}$. By definition, there is some $P\in
{\mathcal P}$ such that ${\mathcal R}=\{S\in{\mathcal P}|S\subseteq
P\}$. It is easy to see that $i^{-1}({\mathcal R})=P$.

We need to show that $(Y,{\mathcal Q})$ is geometric. $Y$ is a
complete lattice, ordered by set-inclusion, and $\mathcal Q$ is the
set of principal downsets of $Y$. This means that $\widetilde{\mathcal
  Q}$ is the set of ideals in $Y$, and $\overline{\mathcal Q}$ is the
set of closed sets of the weak topology. From
Examples~\ref{TCSexamples}.2, we know that the intersection of these
is $\mathcal Q$.
\end{proof}


This leads to the natural question is what subspaces of a geometric
preconvexity space are geometric.

\begin{proposition}
  If $(X,{\mathcal P})$ is a geometric preconvexity space and
  $A\in{\mathcal P}$, then the restriction $(A,{\mathcal P}|_A)$ is a
  geometric preconvexity space.
\end{proposition}

\begin{proof}
  Since $\mathcal P$ is closed under intersection, ${\mathcal
    P}|_A\subseteq {\mathcal P}$. Now let $C\subseteq A$ be
  both a directed union of sets from ${\mathcal P}|_A$ and an
  intersection of finite unions of sets from ${\mathcal P}|_A$.
  Since ${\mathcal P}|_A\subseteq {\mathcal P}$, $C$ is 
  both a directed union of sets from ${\mathcal P}$ and an
  intersection of finite unions of sets from ${\mathcal P}$. Since
  $(X,{\mathcal P})$ is geometric, it follows that $C\in \mathcal P$,
  and since $C\subseteq A$, we have $C\in{\mathcal P}|_A$ as required.
\end{proof}

On the other hand, closed or convex subspaces of geometric
preconvexity spaces are not necessarily geometric.

\begin{example}
Let $X={\mathbb R}^2\setminus\{(0,0)\}$, and let $$Y=\{(x,y)\in[0,1]^2|(|2x-1|-1)(|2y-1|-1)=0\}$$ be the unit square with one corner at
the origin. It is straightforward to check that $X$ and $Y$, with the
preconvexities coming from closed convex subsets of ${\mathbb R}^2$,
are geometric. However, $X\cap Y$ is a closed subspace of $X$, and a
convex subspace of $Y$, but is not geometric, since the subset
$\{(x,y)\in X\cap Y|x>0\textrm{ or }y=1\}$ is both closed and convex,
but is not closed convex.
\end{example}

\section{Stone Duality}

\subsection{Stone Duality for Topological Spaces}

In this section, we review Stone duality for topological spaces. While
a lot of what we review is well-known, some parts are written from an
unusual perspective, and are not as well-known as they might be.

Given a topological space, the collection of closed sets form a
coframe. (Many authors refer to the frame of open sets, but for our
purposes the closed sets are more natural, and since we are not
considering 2-categorical aspects, it does not matter since
$\Coframe=\Frame\co$.) Furthermore, the inverse image of a continuous
function between topological spaces is by definition a coframe
homomorphism between the coframes of closed spaces. This induces a
functor $C:\Topol\morphism{}\Coframe\op$. Not every coframe arises as
closed sets of a topological space. Coframes that do arise in this way
are called {\em spatial} and are said to ``have enough points''.

In some cases, there can be many topological spaces that have the same
coframe of closed sets. If multiple points have the same closure, then
there is no way to separate them by looking at the coframe of closed
sets. Therefore, we restrict our attention to $T_0$ spaces, where the
function from $X$ to $C(X)$ sending a point to its closure is
injective. If we restrict to the spatial coframes, then the functor
$T_0$-$\Topol\morphism{C}\SpatialCoframe$ is a faithful fibration (we
will show this later).

We can recover a $T_0$ topological space from its lattice of closed
sets and from the subset $S\subset C(X)$ consisting of the closures of
singletons. For a coframe $L$, the elements which could arise as
closures of singletons for a topological space corresponding to $L$
are elements that cannot be written as a join of two strictly smaller
elements (called {\em join-irreducible} elements). These are called
the ``points'' of $L$ since they correspond to coframe homomorphisms
$f:L\morphism{}2$, where the 2-element coframe, 2, is the terminal
object in $\Coframe\op$.
If we let $P_L\subseteq L$ be the set of points, a topological space
$X$ corresponds to a coframe $L=C(X)$ with a chosen subset $S\subseteq
P_L$ such that for every $x\in L$, we have $x=\bigvee(S\cap\down x)$
(that is, $S$ is join-dense in $L$). Continuous functions
$X\morphism{g}Y$ correspond to coframe homomorphisms
$C(Y)\morphism{C(g)}C(X)$ whose left adjoint $C(X)\morphism{C(g)^*}C(Y)$
(in the category of order-preserving maps) sends $S\!_X$ to $S\!_Y$.
We can express this left adjoint condition topologically as:
for every $s\in S\!_X$,
$C(g)^{-1}(\down s)$ is a principal downset in $C(Y)$, and the top
element is in $S\!_Y$, where $S\!_X\subseteq P_{C(X)}$ and
$S\!_Y\subseteq P_{C(Y)}$ are the chosen sets of points that
correspond to elements of $X$ and $Y$ respectively.

More formally, let $\SpatialCoframe_*$ be the category of pointed
spatial coframes. Objects are pairs $(L,S)$ where $L$ is a coframe and
$S\subseteq P_L$ is a join dense set of points of $L$ (meaning
$(\forall a\in L)(a=\bigvee(S\cap\down a))$). Morphisms
$(M,T)\morphism{g}(L,S)$ are coframe homomorphisms $M\morphism{g}L$
whose left adjoint $L\morphism{g^*}M$ as order-preserving
homomorphisms restricts to a function $S\morphism{g^*_S}T$.

\begin{proposition}
The category of $T_0$ topological spaces and continuous functions is
equivalent to the category $\SpatialCoframe_*\op$.
\end{proposition}

\begin{proof}
The functor $C:T_0\textrm{-}\Topol\morphism{}\SpatialCoframe_*\op$
sends a topological space $X$ to the pair
$\left(C(X),\left\{\overline{\{x\}}\middle|x\in X\right\}\right)$,
where $C(X)$ is the coframe of closed subsets of $X$. It sends a
continuous function $f:X\morphism{}Y$ to
$f^{-1}:C(Y)\morphism{}C(X)$. We need to show that this is a
homomorphism in $\SpatialCoframe_*$. It is clearly a coframe
homomorphism, so we need to show that for any $x\in X$,
$\bigwedge\left\{t\in \left\{\overline{\{y\}}\middle|y\in Y\right\}
\middle|\overline{\{x\}}\leqslant f^{-1}(t)\right\}\in T$. We will
show that $\bigwedge\left\{t\in \left\{\overline{\{y\}}\middle|y\in
Y\right\} \middle|\overline{\{x\}}\leqslant
f^{-1}(t)\right\}=\overline{\{f(x)\}}$. We need to show that
$\overline{\{x\}}\subseteq f^{-1}\left(\overline{\{f(x)\}}\right)$,
and if $\overline{\{x\}}\subseteq f^{-1}\left(A\right)$ for any closed
$A\subseteq Y$, then $\overline{\{f(x)\}}\subseteq A$. Clearly, $x\in
f^{-1}\left(\overline{\{f(x)\}}\right)$, so
$f^{-1}\left(\overline{\{f(x)\}}\right)$ is a closed set containing
$x$, so $\overline{\{x\}}\leqslant
f^{-1}\left(\overline{\{f(x)\}}\right)$. On the other hand, suppose
$\overline{\{x\}}\leqslant f^{-1}\left(A\right)$. Then $x\in
f^{-1}\left(A\right)$, so $f(x)\in A$, so
$\overline{\{f(x)\}}\leqslant A$. Thus $f^{-1}$ is a morphism in
$\SpatialCoframe_*\op$.

In the opposite direction, the functor $P:\SpatialCoframe_*\op
\morphism{}T_0\textrm{-}\Topol$ sends the pair $(L,S)$ to the
topological space with elements $S$ and closed sets
$\left\{S\cap\down a\middle| a\in L\right\}$. For the morphism
$(L,S)\morphism{f}(M,T)$, we define $T\morphism{f^*}S$ as the
restriction of the left adjoint of $f$ to $T$. By definition of
$\SpatialCoframe_*$, this is a well-defined function.  For any closed
$F\in L$, we have $f^*(t)\in\down F$ if and only if $t\leqslant f(F)$ by
definition, so $(f^*)^{-1}(\down F)=S\cap\down f(F)$ is a closed
subset of $P(M,T)$. Thus $f^*$ is continuous.


Finally, we need to show that the two functors defined above form an
equivalence. For a topological space $X$, we see that $PCX$ has the
same elements as $X$ and closed sets of $PCX$ are of the form
$\down F\!\cap \left\{\overline{\{x\}}\middle|x\in X\right\}$ for $F\in
C(X)$. It is clear that $\overline{\{x\}}\leqslant F$ if and only if
$x\in F$, so closed sets of $PCX$ are exactly closed sets of $X$, so
$PCX\cong X$.

For a coframe $L$ with a chosen subset $S\subseteq L$, we want to show
that $CP(L,S)\cong (L,S)$. By definition, elements of $CP(L,S)$ are
$\left\{\down\!a\cap S\middle| a\in L\right\}$. Since $(\forall a\in
L)(a=\bigvee(\down\!a\cap S))$, it follows that the coframe of
$CP(L,S)$ is isomorphic to $L$. The chosen elements are
$\left\{\overline{\{s\}}\middle| s\in S\right\}$, where
$\overline{\{s\}}$ is the closure of $\{s\}$ in $P(L,S)$. Closed sets
of $P(L,S)$ are of the form $\down\!a\cap S$ for $a\in L$, so in
particular $\overline{\{s\}}=\down\!s\cap S$. This clearly induces an
isomorphism $(L,S)\cong CP(L,S)$. 
\end{proof}

\begin{proposition}
The forgetful functor $U:\SpatialCoframe_*\morphism{}\SpatialCoframe$
is an op-fibration.
\end{proposition}

\begin{proof}
We need to show that for any $(M,T)$ in $\SpatialCoframe_*$ and any coframe
homomorphism $M\morphism{f}L$, where $L$ is spatial, there is
some $S\subseteq P_L$ that makes $f$ a cocartesian morphism $(M,T)\morphism{f}(L,S)$.

Recall that $f$ is a homomorphism in $\SpatialCoframe_*$ if and only
if its left adjoint $f^*$ sends $S$ to $T$, or equivalently
$S\subseteq (f^*)^{-1}(T)$. We will show that if
$S=(f^*)^{-1}(T)$ then $f$ is cocartesian. Let $L\morphism{g}K$ be a
coframe homomorphism, and let $R\subseteq P_K$ be such that
$(M,T)\morphism{gf}(K,R)$ is a homomorphism in
$\SpatialCoframe_*$. Then for any $r\in R$, we have $f^*g^*(r)\in T$,
so $g^*(r)\in(f^*)^{-1}(T)=S$, as required. Conversely, if $S\ne
(f^*)^{-1}(T)$, then the identity function is a coframe homomorphism
$S\morphism{}S$, and the composite $1\circ f$ is a pointed coframe
homomorphism $(M,T)\morphism{}(L,(f^*)^{-1}(T))$, but the identity
coframe homomorphism is not a pointed coframe homomorphism $(L,S)\morphism{}(L,(f^{*})^{-1}(T))$.
%
%

Thus, $f$ is cocartesian if and only if $S=(f^*)^{-1}(T)$. From this,
it is clear that if $S=(f^*)^{-1}(T)$ makes $f$ into a morphism in
$\SpatialCoframe_*$, then it is the unique cocartesian lifting of $f$.

Recall that $(M,T)\morphism{f}(L,S)$ is a morphism in
$\SpatialCoframe_*$ if for every $s\in S$, $f^*(s)\in T$. This is
clearly the case when $S=(f^*)^{-1}(T)$.
\end{proof}

An alternative approach to modelling the category of $T_0$ topological
spaces is via strictly zero-dimensional biframes.  A {\em biframe}
\cite{BanaschewskiBrummerHardie} consists of a frame $L_0$ with two
chosen subframes $L_1$ and $L_2$. The biframe $(L_0,L_1,L_2)$ is
strictly zero-dimensional if every element of $L_1$ is complemented in
$L_0$, and the complement is in $L_2$. A particular example is the
Skula biframe of a topological space $X$, where $L_1$ is the frame of
open sets, $L_2$ is the frame of subsets generated by the closed sets,
and $L_0$ is the frame of subsets generated by $L_1$ and $L_2$ (which
are the open sets in the Skula topology~\cite{Skula}.) The functor
that sends a topological space to the Skula biframe is one half of an
equivalence between the category of $T_0$ topological spaces and the
category of strictly zero-dimensional biframes~\cite{Manuell}. In this
biframe, $L_2$ is actually a completely distributive lattice, so also
a coframe, so the whole Skula biframe can be represented by the
coframe inclusion $(L_1)\op\mono L_2$. In this case $L_2$ is the
coframe generated by the closures of all elements of $X$. That is, for
$(L,S)\in\ob\SpatialCoframe_*$, this is the inclusion
$L\mono DS$, where $D$ is the downset functor. The condition that
$f^*$ restricts to a function $S\morphism{f^*}T$ means that the
inverse image function $Df^*:DT\morphism{}DS$ is a complete lattice
homomorphism. Furthermore, the diagram

\hfil\xymatrix{M \ar[r]^ f \ar@{ >->}[d] & L \ar@{ >->}[d]\\
  DT \ar[r]_{Df^*}  &  DS}

\noindent commutes, where the inclusion $M\mono DT$ sends  $x\in M$ to
$T\cap\down\!x$. To see that the diagram commutes, note that $Df^*$ sends $T\cap\down\!x$ to $$\{s\in
S|f^*(s)\in T\cap\down\!x\}=\{s\in S|f^*(s)\leqslant x\}=\{s\in
S|s\leqslant f(x)\}$$ which is the image of $f(x)$ in the inclusion
$L\mono DS$. The condition that
$S\subseteq L$ means that all totally compact elements of $DS$ are in
$L$, so every element of $DS$ is the join of elements in $L$. We will
refer to such lattice inclusions as dense. Thus the category of $T_0$
topological spaces is equivalent to the category of dense inclusions
of spatial coframes into totally compactly generated completely
distributive lattices.

For all of these representations of $T_0$ topological spaces, the
fibration $$T_0\textrm{-}\Topol\morphism{C}\SpatialCoframe\op$$ is a forgetful
functor. Furthermore, we see that the fibres are additional structure
on the coframe, and that they are partially ordered by inclusion of
this additional structure. Furthermore, every fibre has a top element,
which gives an adjoint to this fibration sending a spatial coframe to
the top element of the fibre over it. (In fact, this adjoint extends
to all coframes, because spatial coframes are reflective in all
coframes). The top elements of the fibres are exactly the sober spaces.

Not all fibres have bottom elements. However, a large number of the
fibres of the fibration
$T_0\textrm{-}\Topol\morphism{C}\SpatialCoframe\op$ do have bottom
elements and are actually complete Boolean algebras. This is probably
easiest to see from the representation as coframes with a chosen set
of elements which are closures of points of the topological space. If
$S\!_0$ is the smallest such set of closed sets that can arise as
closures of points, and $S\!_1$ is the largest set, then any set
between $S\!_0$ and $S\!_1$ is a valid set of points, making the poset
of possible sets of points isomorphic to the Boolean algebra
$P(S\!_1\setminus S\!_0)$. The topological spaces that can occur as
the bottom elements of fibres are spaces where the closure of every
point cannot be expressed as a union of closed sets not containing
that point. That is, for every $x\in X$, $\overline{\{x\}}\setminus
\{x\}$ is closed. Spaces with this property are called $T\!_D$
spaces~\cite{AullThron}.

Clearly, all $T\!_1$ spaces are $T\!_D$ because in a $T\!_1$ space
$\overline{\{x\}}\setminus \{x\}=\emptyset$ is closed. However, even
if we restrict to $T\!_1$ spaces and atomic spatial coframes, the
assignment of an atomic spatial coframe to the bottom element in the
corresponding fibre is not functorial, because the adjoint to a
coframe homomorphism between $T\!_D$ spaces does not necessarily
preserve join-indecomposable elements. This is why the focus of
attention in most of the literature has been on sober spaces, rather
than $T\!_D$ spaces. In order to model the morphisms between $T\!_D$
spaces, we need to restrict to coframe homomorphisms whose adjoint
preserves join-indecomposable elements. While most of the topological
spaces of interest are $T\!_D$, many of the fibres of the fibration
$T_0\textrm{-}\Topol\morphism{C}\SpatialCoframe\op$ contain only a
singleton $T\!_0$ topological space, which is therefore both sober and
$T\!_D$. (Several equivalent characterisations of when this occurs are
given in~\cite{Hoffman1977}.) Thus many important topological spaces are sober.

\subsection{Stone Duality for Preconvexity Spaces}\label{SDPreconvex}

There is in many ways, a very similar picture for the category of
preconvexity spaces. Instead of the coframe of closed sets, the
structure that defines the preconvexity spaces is the
complete lattice of preconvex sets ${\mathcal P}$. Because the inverse
image function for a preconvexity space homomorphism preserves
preconvex sets, it induces an inf-homomorphism between the lattices of
preconvex sets. Thus, we have a functor
$\Preconvex\morphism{P}\Inf\op$, where $\Inf$ is the category of
complete lattices with infimum-preserving homomorphisms between them,
sending every preconvexity space to its lattice of preconvex sets, and
every homomorphism to the inverse image function. This has many of the
nice properties of the Stone duality functor
$\Topol\morphism{F}\Coframe\op$.

As in the topology case, there is an equivalent category of
sup-lattices with a set of chosen elements. Let $\TCGPartialSup$ be
the category whose objects are pairs $(L,S)$, where $L$ is a complete
lattice and $S\subseteq L$ is sup-dense, i.e. $(\forall x\in
L)(x=\bigvee (S\cap\down x))$. Morphisms $(L,S)\morphism{f}(M,T)$ in
$\TCGPartialSup$ are sup-homomorphisms $L\morphism{f}M$ such that
$(\forall x\in S)(f(x)\in T)$.

\begin{proposition}\label{TCGPSPrecEquiv}
The categories $\TCGPartialSup$ and $\Preconvex$ are equivalent.
\end{proposition}

\begin{proof}
There is a functor
$\Preconvex\morphism{F}\TCGPartialSup$ given by $F(X,{\mathcal
  P})=\left({\mathcal P},\left\{\overline{\{x\}}\middle|x\in
X\right\}\right)$ on objects and $F(f)\vdash f^{-1}$ on morphisms,
where the adjoint is as a partial order homomorphism and exists
because $f^{-1}$ is an inf-homomorphism. To show this is well-defined,
since $F(f)$ is a right adjoint, it is a sup-homomorphism, defined by
$F(f)(A)=\bigcap\{B\in{\mathcal P}'| A\subseteq f^{-1}(B)\}$. In
particular, if $A=\overline{\{x\}}$, then
$$F(f)(A)=\bigcap\{B\in{\mathcal P}'|x\in
f^{-1}(B)\}=\bigcap\{B\in{\mathcal P}'|f(x)\in
B\}=\overline{\{f(x)\}}$$
To complete the proof that $F$ is well-defined, we need to show that
$\left\{\overline{\{x\}}\middle| x\in X\right\}$ is sup-dense in
$\mathcal P$. For any $P\in{\mathcal P}$, and any $x\in P$, we have
$\overline{\{x\}}\subseteq P$. Thus
$P=\bigcup\left\{\overline{\{x\}}\middle|x\in P\right\}$ as required.
Thus $F$ is well-defined, and functoriality
is obvious. In the other direction, there is a functor
$G:\TCGPartialSup\morphism{}\Preconvex$ given by
$G(L,S)=(S,\{S\cap\down x|x\in L\})$ on objects and $G(f)(s)=f(s)$
on morphisms. To show well-definedness, we need to show that if
$(L,S)\morphism{f}(M,T)$ is a morphism of $\TCGPartialSup$, then
$G(f)$ is a preconvexity homomorphism. That is, for any $s\in S$, we
have $f(s)\in T$, and for any $m\in M$, $G(f)^{-1}(T\cap \down
m)=S\cap\down x$ for some $x\in L$. The first condition is by
definition of a homomorphism. Since $f$ is a sup-homomorphism, it has
a right adjoint $f^*$ given by $f^*(m)=\bigwedge\{x\in L|f(x)\geqslant
m\}$. If we let $x=f^*(m)$, then 
$$G(f)^{-1}(T\cap\down m)=\{s\in S|f(s)\leqslant m\}=\{s\in
S|s\leqslant f^*(m)=x\}=S\cap\down x$$
which gives the required homomorphism property. Finally, we want to
show that $F$ and $G$ form an equivalence of categories. For a
preconvexity space $(X,{\mathcal P})$, we have that
$$GF(X,{\mathcal P})=G\left({\mathcal P},\left\{\overline{\{x\}}|x\in
X\right\}\right)=\left(\left\{\overline{\{x\}}|x\in
X\right\},\left\{\left\{\overline{\{x\}}|x\in
X\right\}\cap\down P \middle|P\in \mathcal P\right\}\right)$$
It is obvious that the function sending $x$ to $\overline{\{x\}}$ is
a natural isomorphism of preconvexity spaces. In the other direction,
for $(L,S)\in\ob(\TCGPartialSup)$, we have
  $$FG(L,S)=F(S,\{S\cap\down x|x\in L\})=\left(\{S\cap\down x|x\in
L\},\left\{\overline{\{s\}}\middle|s\in S\right\}\right)$$ It is easy
to see that for $s\in S$, $\overline{\{s\}}=S\cap\down s$, so the
function $L\morphism{i}\{S\cap\down x|x\in L\}$ given by $i(x)=
S\cap\down x$ is easily seen to be an isomorphism in
$\TCGPartialSup$. Thus we have shown the equivalence of
categories. Under this equivalence (and the adjoint isomorphism
$\Sup\cong\Inf\op$), the functor $\Preconvex\morphism{P}\Inf\op$ becomes
the forgetful functor $\TCGPartialSup\morphism{U}\Sup$ sending $(L,S)$
to $L$.
\end{proof}

\begin{proposition}
  The forgetful functor $\TCGPartialSup\morphism{U}\Sup$ is a fibration.
\end{proposition}

\begin{proof}
Firstly, we show that cartesian morphisms for $U$ are morphisms of the
form $(L,S)\morphism{f}(M,T)$ such that $S=f^{-1}(T)$. If
$S=f^{-1}(T)$, then for any object $(N,R)$ of $\TCGPartialSup$ and any
sup-homomorphism $N\morphism{g}L$ such that $fg$ is a morphism of
$\TCGPartialSup$, we will show that $g$ is a morphism of 
$\TCGPartialSup$. To do this, we need to show that for any $r\in R$,
$g(r)\in S$. Since $fg$ is  a morphism of
$\TCGPartialSup$, we have $fg(r)\in T$. Thus $g(r)\in f^{-1}(T)=S$ as
required. Conversely, suppose $(L,S)\morphism{f}(M,T)$ is
cartesian. We want to show that $S=f^{-1}(T)$. Let
$i:L\morphism{}L$ be the identity function. This is a
sup-homomorphism, and the composite $fi$ is by definition a morphism
$(L,f^{-1}(T))\morphism{fi}(M,T)$ in $\TCGPartialSup$, so for $f$ to
be cartesian, $i$ must be a morphism $(L,f^{-1}(T))\morphism{i}(L,S)$
in $\TCGPartialSup$. This gives $f^{-1}(T)\subseteq S$, and the fact
that $f$ is a morphism in $\TCGPartialSup$ gives $S\subseteq
f^{-1}(T)$. Thus we have shown that $f$ is cartesian if and only if
$S=f^{-1}(T)$. 

To show that $U$ is a fibration, we need to show that for any
sup-homomorphism $L\morphism{g}M$, and any object $(M,T)$ of
$\TCGPartialSup$ over $M$, there is a unique lifting of $g$ to a
cartesian homomorphism $(L,S)\morphism{\hat{g}}(M,T)$ in
$\TCGPartialSup$. By definition $\hat{g}$ must be the sup-homomorphism
$g$, and we must have $S=g^{-1}(T)$. It is clear that this gives a
unique choice for $S$.
\end{proof}

As in the topological case, it is easy to see that the fibres of the
fibration $U$ are partial orders. Each fibre clearly has a top element
setting $S=L$. This gives

\begin{proposition}\label{PreconvexityAdjoint}
The preconvex set lattice functor $\Preconvex\morphism{U}\Sup$ has a
right adjoint $\Sup\morphism{P}\Preconvex$.
\end{proposition}

\begin{proof}
  The right adjoint $P$ is defined by $P(L)=(L,\{\down x|x\in
  L\})$. That is, it sends the complete lattice $L$ to $L$ with the
  preconvexity where only principal downsets are preconvex. From the
  equivalence, between $\Preconvex$ and $\TCGPartialSup$, this $P$
  sends $L$ to the pair $(L,L)$, which is clearly the top element of
  the fibre of the forgetful functor, $U$. Because $U$ is a fibration,
  the $P$ is its right adjoint as required.
\end{proof}

Bottom elements of the fibres are of the form $(L,S)$ where $S$ is the
smallest subset of $L$ satisfying $(\forall x\in L)(x=\bigvee
S\cap\down x)$. For any $x\in L$, if we can find a downset
$D\subseteq L$ with $\bigvee D=x$ and $x\not\in D$, then clearly if
$(L,S)\in\ob\TCGPartialSup$, then
$(L,S\setminus\{x\})\in\ob\TCGPartialSup$, so if there is a minimum
set $S$, then we cannot have $x\in S$. Conversely, if the only downset
whose supremum is $x$ is the principal downset $\down x$, then for
any $(L,S)$ in $\TCGPartialSup$, we must have $x\in S$. Thus, if there
is a smallest element of the fibre above $L$, it must be given by
$(L,S)$, where
$$S=\left\{x\in L\middle|(\forall D\subseteq L)\left(\left(\bigvee D=x\right)\Rightarrow x\in D)\right)\right\}$$
This is similar to the total compactness condition on elements of a
sup-lattice, but an element $x$ is called totally compact if it
satisfies $$(\forall D\subseteq L)\left(\left(\bigvee D\geqslant
x\right)\Rightarrow (\exists y\in D)(x\leqslant y)\right)$$ which is a
stronger condition.

As in the topological case, when bottom elements of the fibre exist,
they are usually the spaces of greatest interest. For example, spaces
where every singleton set is preconvex are always the bottom elements
of the corresponding fibre. However, for this fibration, the fibres
are very rarely singletons, so the top elements of the fibres are not
of as much interest as in the topological case.

It is also worth noting that we have the chain of adjunctions
$$
\xymatrix@1{{\TC} \ar@<+0.75ex>[r]^{CC} \ar@{}[r]|\bot & {\Preconvex}
  \ar@<+0.75ex>[l]^{IS} \ar@<+0.75ex>[r]^(0.6){U} \ar@{}[r]|(0.6)\bot & {\Sup} \ar@<+0.75ex>[l]^(0.4){T}}$$
which gives an adjunction between the category of topological
convexity spaces and the category of sup-lattices. This adjunction
sends a topological convexity space $(X,{\mathcal F},{\mathcal C})$ to
the lattice of sets ${\mathcal F}\cap{\mathcal C}$ ordered by set
inclusion, and a topological convexity space homomorphism to the left
adjoint of its inverse image. The right adjoint sends a sup-lattice
$L$ to the topological convexity space $(L,{\mathcal S},{\mathcal
  I})$, where ${\mathcal S}$ is the set of weak-closed subsets of
$L$, namely intersections of finitely-generated downsets in $L$, and
${\mathcal I}$ is the set of ideals in $L$.

\begin{theorem}\label{MainTheorem}
There is an adjunction between the category of topological
convexity spaces and the category of sup-lattices. The left adjoint
sends a topological convexity space $(X,{\mathcal F},{\mathcal C})$ to
the lattice ${\mathcal F}\cap {\mathcal C}$ of closed convex
sets, ordered by inclusion, and a topological convexity space homomorphism  
$X\morphism{f}Y$ to the adjoint of its inverse image function. The
right adjoint sends a sup-lattice $L$ to the topological convexity
space $(L,{\mathcal S},{\mathcal I})$ from
Example~\ref{TCSexamples}(2), and a sup-homomorphism $L\morphism{f}K$
to $f$ viewed as a topological convexity space homomorphism.
\end{theorem}

\begin{proof}
It is straightforward to check that these functors are the composites
of the adjunctions 

$$
\xymatrix@1{{\TC} \ar@<+0.75ex>[r]^{CC} \ar@{}[r]|\bot & {\Preconvex}
  \ar@<+0.75ex>[l]^{IS} \ar@<+0.75ex>[r]^(0.6){U} \ar@{}[r]|(0.6)\bot & {\Sup} \ar@<+0.75ex>[l]^(0.4){P}}$$

shown in Proposition~\ref{AdjointTCPreconvex} and Proposition~\ref{PreconvexityAdjoint}.
\end{proof}

\begin{remark}
  In the abstract, we described the relation between topological
  convexity spaces and inf-lattices as an extension to the Stone
  duality between topological spaces and coframes. Any topological
  space is a topological convexity space with the discrete convexity,
  where all sets are convex. Similarly, the category of coframes is a
  subcategory of the category of inf-lattices. The following diagram
  commutes:

  \hfil\xymatrix{{\Topol} \ar[d]_C \ar@{ >->}[r]^D & {\ConvexTop}
    \ar[d]^{UCC} \\ {\Coframe\op} \ar@{ >->}[r] & {\Inf\op}}

  \noindent However, the adjoint to $UCC$ does not restrict to an
  adjoint to the closed set coframe functor, $C$, because $UCC(L)$ is
  not in general a topological space, even if $L$ is a coframe. Thus
  only the fibration part is truly an extension, and the duality is
  not an extension.
  
\end{remark}

\subsection{Distributive Partial-Sup Lattices}

The equivalence $\Preconvex\cong\TCGPartialSup$ is based on previous
work~\cite{PartialSup}. We present this work in a more abstract
framework here. The idea is that for a preconvexity space $(X,\mathcal
P)$, the sets in $\mathcal P$ are partially ordered by inclusion. This
partial order has an infimum operation given by intersection, but
union of sets only gives a partial supremum operation because a union
of preconvex sets is not necessarily preconvex. (Because of the
existence of arbitrary intersections, there is a supremum operation
given by union followed by the induced closure operation, but this
supremum is not related to the structure of the preconvexity
space. Unions of preconvex sets better reflect the structure of the
preconvexity space. We therefore add a partial operation to the
structure to describe these unions where possible.) For a preconvexity
space, the operations are union and intersection, so we have a
distributivity law between the partial join operation and the infimum
operation. This can be neatly expressed by saying that the partial
join structure is actually a partial join structure in the category
$\Inf$.
We define a partial join structure as a partial algebra for the
downset monad. The downset monad exists in the category of partial
orders, and also in the category of inf-lattices.

We begin by recalling the following definitions:

\begin{definition}[\cite{Kock1995}]
  A {\em KZ-doctrine} on a 2-category $\mathcal C$ is a monad $(T,\eta,\mu)$
  on $\mathcal C$ with a modification $\xymatrix@1{T\eta
    \ar@{=>}[r]^{\lambda} & \eta_T}$
such that $\lambda\eta$, $\mu\lambda$ and $\mu T\mu\lambda_T$  are all
identity 2-cells.
\end{definition}


\begin{definition}[\cite{DawsonParePronk2003}]
  A 2-functor ${\mathcal C} \morphism{F} {\mathcal D}$ is {\em
    sinister} if for every morphism $f$ in $\mathcal C$, $F\!f$ has a
  right adjoint in $\mathcal D$.
\end{definition}


In particular, if $F$ is sinister, then it acts on partial maps.

\begin{definition}
A {\em lax partial algebra} for a sinister KZ-doctrine in an
order-enriched category is a partial map
$\xymatrix@1{TX \ar@_{->}[r]^\theta & X}$ such that 

\hfil\xymatrix{X \ar[r]^{\eta} \ar[rd]_{1_X}& TX
  \ar@_{->}[d]^\theta\\&X}

\noindent commutes and there is a 2-cell

\hfil\xymatrix{TTX \ar[r]^{T\theta} \ar[d]_{\mu_X} \ar@{}[dr]|\Leftarrow & TX
  \ar@_{->}[d]^\theta\\ TX \ar@_{->}[r]_\theta &X }%

A {\em homomorphism} of lax partial algebras from $(X,\theta)$ to
$(Y,\tau)$ is a morphism $X\morphism{f}Y$, together with a 2-cell

\hfil\xymatrix{TX \ar@_{->}[d]_\theta \ar[r]^{Tf}
  \ar@{}[rd]|\Rightarrow & TY \ar@_{->}[d]^{\tau} \\ X \ar[r]_f & Y}

\end{definition}

\begin{remark}
It is possible to define lax partial algebras for KZ monads in general
2-categories. However, this requires more careful consideration of
coherence conditions, so to focus on the particular case of
distributive partial sup-lattices, we have restricted attention to
$\Ord$-enriched categories, where $\Ord$ is the category of
partially-ordered sets and order-preserving functions between them.
\end{remark}

\begin{definition}\label{PSLPartialAlgebra}
  A {\em partial sup-lattice} is a lax partial algebra for the sinister KZ
  doctrine $(D,\down,\bigcup)$ in \Ord, where $D$ is the downset
  functor, $\down_X$ is the function sending an element $x\in X$ to
  the principal downset it generates, and $\bigcup_X:DDX\morphism{}DX$
  sends a collection of downsets to its union.

  A {\em distributive partial sup-lattice} is a lax partial algebra for
  the sinister KZ doctrine $(D,\down,\bigcup)$ in \Inf, where $D$ is
  the downset functor, $\down_X$ is the function sending an element
  $x\in X$ to the principal downset it generates, and
  $\bigcup_X:DDX\morphism{}DX$ sends a collection of downsets to its
  union.
\end{definition}  

The definition given in~\cite{PartialSup} is

\begin{definition}[\cite{PartialSup}]\label{PSL}
A {\em partial sup lattice} is a pair $(L,J)$ where $L$ is a complete
lattice, $J$ is a collection of downsets of $L$ with the following
properties:

\begin{itemize}
\item $J$ contains all principal downsets.

\item $J$ is closed under arbitrary intersections.

\item If $A\in J$ has supremum $x$, then any downset $B$ with $A\subseteq
  B\subseteq\down x$ has $B\in J$.
  
\item if ${\mathcal A}\subseteq J$ is down-closed, $Y\in J$ has $\bigvee Y=x$ and for
  any $a\in Y$, there is some $A\in{\mathcal A}$ with $\bigvee
  A\geqslant a$, then there is some $B\subseteq\bigcup{\mathcal A}$
  with $B\in J$ and $\bigvee B\geqslant x$.

\end{itemize}

A partial sup-lattice is
{\em distributive} if for any ${\mathcal D}\subseteq J$,
we have $\bigwedge \left\{\bigvee D|  D\in{\mathcal D}\right\}=\bigvee\bigcap
{\mathcal D}$.

An inf-homomorphism $L\morphism{f}M$ is a {\em partial sup-lattice
homomorphism} $(L,J)\morphism{f}(M,K)$ if for any $A\in J$,
$\down\{f(a)|a\in A\}\in K$, and $\bigvee \down\{f(a)|a\in
A\}=f\left(\bigvee A\right)$.
\end{definition}

\begin{proposition}
  Definitions~\ref{PSLPartialAlgebra} and~\ref{PSL} give equivalent
  definitions of distributive partial sup lattices.
\end{proposition}

\begin{proof}
  We need to show that if $DL\pmorphism{\theta}L$ is a lax partial algebra
  for the downset monad in $\Inf$, then there is some $J\subseteq DL$
  satisfying the conditions of Definition~\ref{PSL}. We will show that
  setting $J$ as the domain of the partial algebra morphism
  $DL\pmorphism{\theta}L$ works.

  From the unit condition

  \hfil\xymatrix{L \ar[r]^{\down} \ar[rrd]_{1_L} & DL & J \ar@{ >->}[l]
    \ar[d]^{\theta}\\ && L}

  \noindent we have that all principal downsets must be contained in
  $J$. This allows us to show that $\theta$ is the join whenever it is
  defined. For $A\in J$, if $x=\bigvee A$, then $A\leqslant \down x$
  in $J$, and for any $a\in A$, we have $\down a\leqslant A$ in
  $J$. Since $\theta$ is order-preserving, this gives
  $a=\theta\left(\down a\right)\leqslant \theta(A)\leqslant
  \theta\left(\down x\right)=x$, so $\theta(A)$ is an upper bound of
  $A$, and is below $x=\bigvee A$. Thus $\theta(A)=\bigvee A$.
  Since the inclusion $J\mono DL$ is an inf-homomorphism, we get
  that $J$ is closed under arbitrary intersections. Suppose $A\in J$
  has supremum $x$, and $B\in DL$ satisfies $A\subseteq
  B\subseteq\down x$. We want to show that $B\in J$.

  The lax partial algebra condition gives

\hfil\xymatrix{DDL \ar[dd]_\bigcup & DJ \ar@{ >->}[l] \ar[r]^{D\theta} & DL
  \\
  & \geqslant & J \ar@{ >->}[u] \ar[d]^\theta \\
DL & J \ar@{ >->}[l] \ar[r]^\theta & L}

\noindent In particular, since $A\in J\cap\down B$, we have
$D\theta(J\cap\down B)=\down x$, and since $\theta(\down x)=x$ is
defined, we have that the upper composite partial morphism is defined
on $J\cap\down B$. For the lower composite, we have
$\bigcup(J\cap\down B)=B$, so for the lower composite to be defined,
we must have $B\in J$.

Finally if ${\mathcal A}\in DJ$, $Y\in J$ has $\bigvee Y=x$ and for
any $a\in Y$, there is some $A_a\in{\mathcal A}$ with $\bigvee
A_a\geqslant a$, then clearly $A_a\cap\down a\in J$, and since
$J\morphism{\bigvee}L$ is an inf-homomorphism, $\bigvee
(A_a\cap\down a)=a$. Thus, setting ${\mathcal
  B}=\down\left\{A_a\cap\down a|a\in Y\right\}$ gives
$D\theta\left({\mathcal B}\right)=Y$, so the upper composite is defined
for $\mathcal B$, and is equal to $x$. Thus, the lower composite gives
$B=\bigcup {\mathcal B}\in J$ with $\theta(B)=x$, which proves the
last condition.

Conversely, suppose that $(L,J)$ is a distributive partial sup lattice
as in Definition~\ref{PSL}. We want to show that $\xymatrix@1{DL & J
  \ar@{ >->}[l] \ar[r]^\bigvee & L}$ is a lax partial algebra for the
downset KZ monad. That is, we want to show that 

\hfil\xymatrix{X \ar[r]^{\down} \ar[rd]_{1_X}& DX
  \ar@_{->}[d]^\theta\\&X}

\noindent commutes, and 

\hfil\xymatrix{DDX \ar[r]^{D\theta} \ar[d]_{\bigcup}
  \ar@{}[rd]|\geqslant & DX
  \ar@_{->}[d]^\theta\\ DX \ar@_{->}[r]_\theta &X }

\noindent We expand the partial morphisms to get the following diagrams

\hfil\xymatrix{L \ar[r] \ar@{=}[d]  & J \ar@{ >->}[d] \ar[ddr]^{\bigvee} \\L \ar[r]^{\down} \ar[rrd]_{1_L}& DL &  \\ & &L}%
\hfil\xymatrix{DDL  \ar[dd]_{\mu_L} & DJ \ar@{ >->}[l] \ar[r]^{D\bigvee} & DL
  \\ & \geqslant &J \ar@{ >->}[u] \ar[d]^\bigvee & \\ DL & J \ar@{ >->}[l] \ar[r]_\bigvee &L }

\noindent The first diagram commutes because $J$ contains all
principal downsets. For the second diagram, if the upper-right
composite of the diagram is defined for $\mathcal A$, then ${\mathcal
  A}\in DJ$ and $Y=\down\left\{\bigvee A\middle|A\in{\mathcal
  A}\right\}\in J$. By definition, for every $a\in Y$, there is
some $A_a\in{\mathcal A}$ such that $\bigvee A_a\geqslant a$. By the
fourth condition in Definition~\ref{PSL}, there is some
$B\subseteq\bigcup{\mathcal A}$, with $B\in J$ and $\bigvee B\geqslant
\bigvee Y$. Now by definition of $Y$, we have $\bigvee Y=\bigvee
\left(\bigcup {\mathcal A}\right)$, so $\bigvee B=\bigvee Y=\bigvee
\left(\bigcup {\mathcal A}\right)$. Now by the third condition of
Definition~\ref{PSL}, it follows that $\bigcup {\mathcal A}\in J$, so
the lower-left composite is defined for ${\mathcal A}$, giving the
required inequality of partial maps.
\end{proof}

\begin{proposition}
  The definition of distributive partial sup-lattice homomorphisms
  given in Definition~\ref{PSL} is equivalent to a lax partial algebra
  homomorphism between lax partial algebras.
\end{proposition}

\begin{proof}
  Because $\theta$ is the restriction of the supremum operation, the
  partial homomorphism condition is exactly that $J$ factors through
  the pullback

\hfil\xymatrix{K^* \ar[r] \ar@{ >->}[d] & K \ar@{ >->}[d] \\ DL \ar[r]_{Df} & DM}

\noindent and for any $A\in J$, $f\left(\bigvee A\right)\leqslant
\bigvee\down\{f(a)|a\in A\}$.

The pullback is given by $K^*=\{A\in DL | Df(A)\in K\}$. Thus the
inclusion is equivalent to the condition for any $A\in J$,
$\down\{f(a)|a\in A\}\in K$.

Since $f$ is order-preserving, for $a\in
A$, we have that $f(a)\leqslant f\left(\bigvee A\right)$, so
$f\left(\bigvee A\right)$ is an upper bound for $\down\{f(a)|a\in
A\}$, and thus $\bigvee\down\{f(a)|a\in A\}\leqslant f\left(\bigvee
A\right)$. Thus the second condition that $\bigvee \down\{f(a)|a\in
A\}\geqslant f\left(\bigvee A\right)$ is equivalent to 
$\bigvee\down\{f(a)|a\in A\}= f\left(\bigvee A\right)$ as required.

\end{proof}

\begin{definition}\label{TCGPSL}
  An element $a$ of a partial sup-lattice $(L,J)$ is {\em totally
    compact} if for any downset $D\in J$, $\bigvee D\geqslant
  a\Rightarrow a\in D$. A partial sup-lattice $(L,J)$ is {\em totally
    compactly generated} if for any $x\in L$, there is some
  $C\subseteq L$ such that every $c\in C$ is totally compact, and so
  that $\down C\in J$ and $\bigvee C=x$.
\end{definition}

\begin{proposition}
The full subcategory of totally compactly generated distributive
partial sup-lattices and partial sup-lattice homomorphisms is
equivalent to the category $\TCGPartialSup\op$ defined at the start of
Section~\ref{SDPreconvex}.
\end{proposition}

\begin{proof}
Given a totally compactly generated
distributive partial sup-lattice $(L,J)$, let $K\subseteq L$ be the
set of totally compact elements of $(L,J)$. Then $(L,K)$ is an element
of $\TCGPartialSup$. Conversely, for $(L,S)\in\ob\TCGPartialSup$, let
$J=\left\{D\in DL\middle|S\cap\down\left(\bigvee
D\right)\subseteq D\right\}$ be the set of downsets of $L$ that
contain all totally compact elements below their supremum. It is clear
that performing these two constructions gives an isomorphic
structure. To show an equivalence of categories, we need to show that
$L\morphism{f}M$ is a distributive partial sup-lattice homomorphism if
and only if it is a morphism in $\TCGPartialSup\op$. Since
distributive partial sup-lattice homomorphisms preserve infima, they
have left adjoints. If $f$ is a partial sup-lattice homomorphism, and
$f^*$ is its left adjoint, then $f^*$ is a sup-homomorphism, and for
any totally compact $a\in M$, if $B\in J$ has $\bigvee B\geqslant
f^*(a)$, then the adjunction gives $f\left(\bigvee B\right)\geqslant
a$. Since $f$ is a partial sup-homomorphism, we have $f\left(\bigvee
B\right)=\bigvee\{f(b)|b\in B\}\geqslant a$. As $a$ is totally
compact, we must have $a\leqslant f(b)$ for some $b\in B$. By the
adjunction, this gives $f^*(a)\leqslant b$. Thus we have shown that if
$B\in J$ has $\bigvee B\geqslant f^*(a)$, then $f^*(a)\in B$. That is,
$f^*(a)$ is totally compact.

Conversely, if $g$ is a sup-homomorphism between totally compactly
generated distributive partial sup-lattices, that preserves totally
compact elements, then its right adjoint is a partial
sup-homomorphism, since if $B\in J$ has $\bigvee B=x$, then
 if $a\leqslant g^*(x)$ is totally compact, then $g(a)\leqslant x$ is
 also totally compact, so $g(a)\in B$. It follows that $a\in
 \down\{g^*(b)|b\in B\}$, so $\bigvee \down\{g^*(b)|b\in B\}=g^*(x)$
 as required.
\end{proof}

\section{Final Remarks and Future Work}

We have extended the fibration from the Stone duality between
topological spaces and the opposite of spatial coframes to a fibration
between topological convexity spaces and sup-lattices (the opposite of
inf-lattices). As in the topological Stone duality, this fibration has
a right adjoint. This right adjoint is not an extension of the
topological case. 

In many ways, the theory is nicer in this situation than in the
topological case. For example, there are no non-spatial sup-lattices:
every sup-lattice arises as the closed convex sets of a topological
convexity space. However, in some ways this nicer theory makes the
results less useful, because in topology, the non-spatial locales fill
some problematic gaps in the category of topological spaces. With
every sup-lattice arising as the closed convex sets of a topological
convexity space, there are no new spaces to be added, so we are not
filling the gaps.

Another significant difference between this and Stone duality for
topological spaces is that for the topological Stone duality, many
interesting topological spaces are in the singleton fibres of the
fibration, meaning that the closed set fibration is full and faithful
for these spaces, so we can study the categorical structure of large
classes of interesting topological spaces using the category of
coframes. For topological convexity spaces, there are no singleton
fibres, and the top elements of fibres (on which the fibration is full
and faithful) are not very interesting topological convexity
spaces. The most interesting topological convexity spaces are the
bottom elements of their fibres, and when we restrict the fibration to
these spaces, it is faithful, but not full, meaning that from a
categorical perspective, $\ConvexTop$ and $\Sup$ are not so closely
related.

The adjunction between topological convexity spaces and sup-lattices
factors through the category of preconvexity spaces, or the equivalent
category of totally compactly-generated distributive partial
sup-lattices. The adjunction between topological convexity spaces and
preconvexity spaces is potentially more interesting, with most
interesting topological convexity spaces being fixed-points of the
induced comonad on $\ConvexTop$. We have characterised which
topological convexity spaces are fixed by this comonad in
Proposition~\ref{Teetotal}, and given some important examples in
Proposition~\ref{GeneralTeetotal}. In the opposite direction, for the
question of which preconvexity spaces are fixed by the induced monad
on $\Preconvex$, we have only been able to show this for a few special
cases.

\subsection{Future Work}

The study of topological convexity spaces is an extremely promising
area of research, including some classical geometric examples and also
some very interesting combinatorial examples. The adjunctions from
this paper are likely to prove extremely valuable in the study of
topological convexity spaces. In this section, we discuss a number of
important problems about topological convexity spaces that may be
addressed using these adjunctions.

\subsubsection{Restricting this to a Duality}

In topology, it is often convenient to restrict Stone duality to an
isomorphism of categories between sober topological spaces and spatial
locales. Sober topological spaces can be described in a number of
topologically natural ways. Similarly, spatial locales can be easily
described. It is easy to describe the topological convexity spaces
from this adjunction, as they come directly from lattices. However,
for the intermediate adjunction between topological convexity spaces
and preconvexity spaces, the conditions for fixed points are less
clear. The characterising conditions in Proposition~\ref{Teetotal} are
not particularly natural, while the natural and commonly used
conditions in Proposition~\ref{GeneralTeetotal} exclude a number of
interesting combinatorial examples. A result between these two that
includes the interesting combinatorial examples but also consists of
natural, easy-to-understand conditions would be extremely valuable. In
the other direction, describing the geometric preconvexity spaces is
more challenging, and could lead to a lot of fruitful research.

\subsubsection{Euclidean Spaces}

The motivating examples for topological convexity spaces are real
vector spaces, particularly finite-dimensional ones. The author has nearly
completed a characterisation of these spaces within the category of
topological convexity spaces, which will be presented in another paper.

\subsubsection{Convexity Manifolds}

In differential geometry, a manifold is a space which has a
local differential structure. That is, the space is covered by a
family of open sets, each of which has a local differential
structure. There are examples of spaces with a cover by open subsets
with a local convexity structure. The motivating example here is real
projective space. We cannot assign a global convexity structure to
projective space, but if we remove a line from the projective plane,
then the remaining space is isomorphic to the Euclidean plane, and so
has a canonical convexity space. Furthermore, these convexity spaces
have a certain compatibility condition --- given a subset $C$ of the
intersection which is convex in both convexity spaces, the convex
subsets of $C$ are the same in both spaces. This gives us the outline
for a definition of convexity manifolds. Further work is needed to
identify the Euclidean projective spaces within the category of
convexity manifolds, and to determine what geometric structure is
retained at this level of generality.

\subsubsection{Metrics and Measures}

There are connections between metrics and measures. For example, on
the real line, any metric that induces the usual convexity space
structure corresponds to a monotone function ${\mathbb
  R}\morphism{d}{\mathbb R}$ with 0 as a fixed point. Such a function
naturally induces a measure on the Lebesgue sets of $\mathbb
R$. Conversely, for every measure on the Lebesgue sets of $\mathbb R$,
we obtain a monotone endofunction of $\mathbb R$ by integrating. Thus,
for the real numbers, there is a bijective correspondence between
metrics that induce the usual topological convexity structure and
measures on $\mathbb R$. This property is specific to $\mathbb R$, and
does not generalise to other spaces like ${\mathbb R}^2$.

There is a more general connection between topological convexity
spaces,  sigma algebras, measures and metrics. 

\begin{example}
  Let $(X,{\mathcal B})$ be a $\Sigma$-algebra. There is a topological
  convexity space $({\mathcal B},{\mathcal F},{\mathcal I})$ where
  $\mathcal I$ is the set of intervals in the lattice $\mathcal B$ and
  $\mathcal F$ is the set of collections of measurable sets closed
  under limits of characteristic functions. That is, for
  $B_1,B_2,\ldots\in {\mathcal B}$, say $B$ is the {\em limit} of
  $B_1,B_2,\ldots$ if for any $x\in B$, there is some $k\in{\mathbb
    Z}^+$ such that $x\not\in B_i\Rightarrow i<k$, and for any
  $x\not\in B$, there is some $k\in{\mathbb Z}^+$ such that $x\in
  B_i\Rightarrow i<k$. $\mathcal F$ is the collection of subsets 
  $F\subseteq {\mathcal B}$ such that for any $B_1,B_2,\ldots\in F$,
  if $B$ is the limit of $B_1,B_2,\ldots$, then $B\in F$.

  \begin{proposition}
    If $\mu$ is a finite measure on $(X,{\mathcal B})$ such that no
    non-empty set has measure zero, then $d:{\mathcal B}\times
    {\mathcal B}\morphism{}{\mathbb R}$ given by
    $d(A,B)=\mu(A\triangle B)$, is a metric and induces the
    topological convexity space structure $({\mathcal B},{\mathcal
      F},{\mathcal I})$ or a finer structure. Furthermore, all metrics
    inducing this topological convexity space structure on $\mathcal
    B$ are of this form.
  \end{proposition}

  \begin{proof}
    We have $d(A,A)=\mu(\emptyset)=0$ and $d(A,B)=d(B,A)$, so we need
    to prove the triangle inequality. That is, for $A,B,C\in{\mathcal
      B}$, we have $d(A,C)\leqslant d(A,B)+d(B,C)$. This is clear
    because $A\triangle C\subseteq A\triangle B\cup B\triangle
    C$. Thus $d$ is a metric. To prove that it induces this
    topological convexity structure, we note that
    $d(A,C)=d(A,B)+d(B,C)$ if and only if $A\triangle C= A\triangle
    B\amalg B\triangle C$. This only happens if $A\cap C\subseteq
    B\subseteq A\cup C$, which means that convex sets must be
    intervals. Finally, we need to show the topology from the metric
    is finer than $\mathcal F$. That is, if $B$ is the limit of
    $B_1,B_2,\ldots$, then $d(B_i,B)\rightarrow 0$. By definition,
    $\bigcap_{i=1}^\infty B_i\triangle B=\emptyset$. Thus, we need to
    show that for a sequence $A_i=B_i\triangle B$ of measurable sets
    with empty intersection, $\mu(A_i)\rightarrow 0$. Let
    $C_i=\bigcup_{j\geqslant i}A_j$. Since $A_i\rightarrow \emptyset$,
    we get $\bigcap_{i=1}^\infty C_i=\emptyset$. Since the $C_i$ are
    nested, we have
    $\lim_{i\rightarrow\infty}\mu(C_i)=\mu\left(\bigcap_{i=1}^\infty
    C_i\right)=0$.
    
    To show that every metric is of that form, let $d:{\mathcal
      B}\times{\mathcal B}\morphism{}\mathbb R$ be a metric on
    $\mathcal B$ whose induced topology and convexity are finer than
    $\mathcal F$ and $\mathcal I$ respectively. We want to show that
    there is a finite measure $\mu$ on $(X,\mathcal B)$ such that
    $d(A,B)=\mu(A\triangle B)$. By the convexity, whenever $A=B\amalg
    C$ is a disjoint union, we have
    $d(A,B)+d(B,\emptyset)=d(\emptyset,A)=d(\emptyset,C)+d(C,A)$ and
    $d(B,\emptyset)+d(\emptyset,C)=d(B,C)=d(B,A)+d(A,C)$. It follows
    that $$2d(A,B)+d(B,\emptyset)+d(A,C)=d(A,C)+d(B,\emptyset)+2d(\emptyset,C)$$
    so $d(A,B)=d(C,\emptyset)$. For general $B$, we have $A\cap B$ is
    between $A$ and $B$, so $$d(A,B)=d(A,A\cap B)+d(A\cap
    B,B)=d(\emptyset,A\setminus B)+d(\emptyset,B\setminus
    A)=d(\emptyset,A\triangle B)$$ Thus, if we define
    $\mu(B)=d(\emptyset,B)$, then $d$ is defined by
    $d(A,B)=\mu(A\triangle B)$. We need to show that $\mu$ is a
    measure on $(X,{\mathcal B})$. That is, that if $A$ and $B$ are
    disjoint, we have $\mu(A\cup B)=\mu(A)+\mu(B)$ and if
    $B_1\subseteq B_2\subseteq \cdots$, then $\mu(\bigcup_{i=1}^\infty
    B_i)=\lim_{i=1}^\infty \mu(B_i)$. We have already shown that the
    first of these comes from the convexity. The second comes from the
    topology. Consider the sequence $A_i=\left(\bigcup_{j=1}^\infty
    B_j\right)\setminus B_i$. By the convexity, we have
    $\mu(A_i)=\mu\left(\bigcup_{j=1}^\infty B_j\right)-\mu( B_i)$, so
    it is sufficient to show that $\mu(A_i)\rightarrow 0$, when
    $(A_i)_{i=1}^\infty$ is a decreasing sequence with empty
    intersection. If $(A_i)_{i=1}^\infty$ is a decreasing sequence
    with empty intersection, then for any $x\in X$, we have $(\exists
    k\in{\mathbb Z}^+)(x\not\in A_k)$. Thus $\emptyset$ is a limit of
    $(A_i)_{i=1}^\infty$. Thus we have $\mu(A_i)\rightarrow
    \mu(\emptyset)=0$ as required.
        
  \end{proof}

\end{example}
  
\subsubsection{Sheaves}

A lot of information about topological spaces can be obtained by
studying their categories of sheaves. A natural question is whether a
similar category of sheaves can be constructed for a topological
convexity space. Part of the difficulty here is that the usual
construction of the sheaf category is described in terms of open
sets. However, for topological convexity spaces, closed sets are more
fundamental, so it is necessary to redefine sheaves in terms of closed
sets. This is conceptually strange. One interpretation of sheaves is
as sets with truth values given by open sets. In this interpretation,
closed sets correspond to the truth values of negated statements, such
as inequality. \cite{Ellerman} argues that inequality is a more
fundamental concept for studying lattices of equivalence relations as
a form of logical statement, so a definition of sheaves in terms of
closed sets could be linked to this work. 

\subsection*{Acknowledgements}
This research was supported by NSERC grant RGPIN/4945-2014.

\bibliographystyle{plain}
\bibliography{references}

\end{document}